\documentclass[final,5p,times,8pt]{elsarticle}
\usepackage{revsymb}
\usepackage{amsmath,bm}
\usepackage{xcolor}
\usepackage{upgreek}
\usepackage{cancel}
\usepackage{tikz}
\usetikzlibrary{arrows.meta, positioning}
\usepackage{amsthm}
\usepackage{booktabs}
\usepackage{tabularx}
\usepackage{float}

\makeatletter
\let\c@lofdepth\relax
\let\c@lotdepth\relax
\makeatother
\usepackage{subfigure}
\makeatletter
\renewcommand{\@thesubfigure}{\normalsize(\textbf{\alph{subfigure}})}
\makeatother

\usepackage{natbib}
\biboptions{comma,square,sort&compress}
\usepackage[colorlinks=true]{hyperref}
\def\diag{\mathop{\mathrm{diag}}}
\newcommand\hcancel[2][black]{\setbox0=\hbox{$#2$}%
\rlap{\raisebox{.25\ht0}{\textcolor{#1}{\rule{0.8\wd0}{0.5pt}}}}#2} 
\newcommand\hcancelt[2][black]{\setbox0=\hbox{$#2$}%
\rlap{\raisebox{.25\ht0}{\textcolor{#1}{\hspace{0.3mm}\rule{0.7\wd0}{0.75pt}}}}#2} 
\newtheorem{Theorem}{Theorem}[section]
\newtheorem{Lemma}{Lemma}[section]
\def\simgt{\,\hbox{\lower0.6ex\hbox{$>$}\llap{\raise0.3ex\hbox{$\sim$}}}\,}
\def\simlt{\,\hbox{\lower0.6ex\hbox{$<$}\llap{\raise0.3ex\hbox{$\sim$}}}\,}
\def\simgteq{\,\hbox{\lower0.6ex\hbox{$\ge$}\llap{\raise0.6ex\hbox{$\sim$}}}\,}
\def\simlteq{\,\hbox{\lower0.6ex\hbox{$\le$}\llap{\raise0.6ex\hbox{$\sim$}}}\,}
\def\applteq{\,\hbox{\lower0.6ex\hbox{$\le$}\llap{\raise0.8ex\hbox{$\approx$}}}\,}
\def\applt{\,\hbox{\lower0.6ex\hbox{$<$}\llap{\raise0.5ex\hbox{$\approx$}}}\,}
\DeclareMathAlphabet\mathbfcal{OMS}{cmsy}{b}{n}
\DeclareMathAlphabet{\mathpzc}{OT1}{pzc}{m}{it}
\DeclareMathAlphabet\euscr{U}{eus}{m}{n}

\DeclareMathOperator{\csch}{csch}

\DeclareMathOperator\supp{supp}
\DeclareMathOperator{\ProductLog}{W}
\makeatletter
\def\user@resume{resume}
\def\user@intermezzo{intermezzo}
\newcounter{previousequation}
\newcounter{lastsubequation}
\newcounter{savedparentequation}
\makeatother
\newcommand{\C}[1]{\mathcal{#1}}

\newcommand{\F}[1]{\mathbf{#1}}

\newcommand{\FR}[1]{\mathfrak{#1}}
\newcommand{\MB}[1]{\mathbb{#1}}

\newcommand{\ME}[1]{\euscr{#1}}

\newcommand{\MBG}{\MB{G}}
\newcommand{\MBSG}{\hat{\MBG}}
\newcommand{\MBR}{\mathbb{R}}
\newcommand{\MBC}{\mathbb{C}}

\newcommand{\MBRP}{\MBR^+}
\newcommand{\MBRzer}{\MBR_0}
\newcommand{\MBRczer}{\MBR_{\hcancel{0}}}
\newcommand{\MBRzerP}{\MBRzer^+}
\newcommand{\MBRmhzer}{\MBR_{-1/2}^-}
\newcommand{\MBZ}{\mathbb{Z}}
\newcommand{\MBZP}{\MBZ^+}
\newcommand{\MBZzer}{\MBZ_0}
\newcommand{\MBZzerP}{\MBZzer^+}
\newcommand{\MBZe}{\MBZ_e}
\newcommand{\MBZeP}{\MBZe^+}

\newcommand{\MFI}{\mathfrak{I}}
\newcommand{\MBI}{\mathbb{I}}
\newcommand{\MBJ}{\mathbb{J}}
\newcommand{\MBJP}{\mathbb{J}^+}

\newcommand{\MBN}{\mathbb{N}}
\newcommand{\MFF}{\FR{F}}

\newcommand{\MFC}{\FR{C}}

\newcommand{\FRI}{\FR{I}}

\newcommand{\cancbra}[1]{\hcancel{[}#1\hcancelt{]}}

\newcommand{\hG}{\hat{G}}
\newcommand{\hK}{\hat{K}}

\newcommand{\bmt}{\bm{t}}

\newcommand{\bmz}{\bm{z}}

\newcommand{\bmy}{\bm{y}}
\newcommand{\hFP}{\hat{\F{P}}}

\newcommand{\bmh}{\bm{h}}

\newcommand{\bmone}{\bm{\mathit{1}}}

\newcommand{\hbmt}{\hat\bmt}

\newcommand{\hx}{\hat{x}}

\newcommand{\hchi}{\hat{\chi}}

\newcommand{\hlambdabar}{\hat \lambdabar}
\newcommand{\hvarpi}{\hat \varpi}

\newcommand{\FOmega}{\F{\Omega}}
\newcommand{\IFOmega}{\bm{\Omega}^{\circ}}

\newcommand{\foralla}{\,\forall_{\mkern-6mu a}\,}
\newcommand{\forallaa}{\,\forall_{\mkern-6mu aa}\,}
\newcommand{\foralle}{\,\forall_{\mkern-6mu e}\,}
\newcommand{\foralls}{\,\forall_{\mkern-6mu s}\,}
\newcommand{\forallS}{\,\forall_{\mkern-4mu rs}\,}
\newcommand{\forallL}{\,\forall_{\mkern-4mu rl}\,}
\newcommand{\Def}[1]{\text{Def}\left(#1\right)}

\newcommand{\trp}[1]{\text{trp}\left(#1\right)}

\newcommand{\Diag}[1]{\diag\left(#1\right)}

\newcommand{\sigmabar}{\mathord{\sigma\kern-0.6em\raisebox{-1ex}{$\bar{\phantom{\sigma}}$}}}
\newcommand{\sigmabarmax}{\mathord{\sigma_{\max}\kern-1.9em\raisebox{-0.8ex}{$\bar{\phantom{\sigma}}$}}\;\;\;\;\;}
\newcommand{\sigmabarmin}{\mathord{\sigma_{\min}\kern-1.75em\raisebox{-0.8ex}{$\bar{\phantom{\sigma}}$}}\;\;\;\;\;}
\newcommand{\RLI}[3]{\,\FRI_{#1}^{#2}#3}
\newcommand{\CapD}[3]{\,{}^{c}D_{#1}^{#2}#3}

\newcommand{\RLIEN}[3]{\,{}_{n_q,\lambda_q}^{n,\lambda,E}\FRI_{#1}^{#2}#3}

\newcommand{\RLIM}[2]{\,{}^{E}\F{Q}_{#1}^{#2}}

\newcommand{\hRLIM}[2]{\,{}^{E}\hat{\F{Q}}_{#1}^{#2}}
\newcommand{\EE}[2]{{#1}^{\circ #2}}
\newcommand{\Part}[1]{\partial_{#1}}
\newcommand{\NPart}[2]{\partial_{#1}^{\,#2}}
\usepackage{mathtools}

\MakeRobust{\eqref}
\errorcontextlines\maxdimen

\makeatletter
    \newcommand*{\algrule}[1][\algorithmicindent]{\makebox[#1][l]{\hspace*{.5em}\thealgruleextra\vrule height \thealgruleheight depth \thealgruledepth}}%
\newcommand*{\thealgruleextra}{}
\newcommand*{\thealgruleheight}{.75\baselineskip}
\newcommand*{\thealgruledepth}{.25\baselineskip}

\newcount\ALG@printindent@tempcnta
\def\ALG@printindent{%
    \ifnum \theALG@nested>0
        \ifx\ALG@text\ALG@x@notext
        \else
            \unskip
            \addvspace{-1pt}
            \ALG@printindent@tempcnta=1
            \loop
                \algrule[\csname ALG@ind@\the\ALG@printindent@tempcnta\endcsname]%
                \advance \ALG@printindent@tempcnta 1
            \ifnum \ALG@printindent@tempcnta<\numexpr\theALG@nested+1\relax
            \repeat
        \fi
    \fi
    }%
\usepackage{etoolbox}

\newbox\statebox
\newcommand{\myState}[1]{%
    \setbox\statebox=\vbox{#1}%
    \edef\thealgruleheight{\dimexpr \the\ht\statebox+1pt\relax}%
    \edef\thealgruledepth{\dimexpr \the\dp\statebox+1pt\relax}%
    \ifdim\thealgruleheight<.75\baselineskip
        \def\thealgruleheight{\dimexpr .75\baselineskip+1pt\relax}%
    \fi
    \ifdim\thealgruledepth<.25\baselineskip
        \def\thealgruledepth{\dimexpr .25\baselineskip+1pt\relax}%
    \fi
    \State #1%
    \def\thealgruleheight{\dimexpr .75\baselineskip+1pt\relax}%
    \def\thealgruledepth{\dimexpr .25\baselineskip+1pt\relax}%
}
\makeatletter
\newcommand{\oset}[3][0ex]{%
  \mathrel{\mathop{#3}\limits^{
    \vbox to#1{\kern-2\ex@
    \hbox{$\scriptstyle#2$}\vss}}}}
\makeatother

\makeatletter
\def\ps@pprintTitle{%
  \let\@oddhead\@empty
  \let\@evenhead\@empty
  \let\@oddfoot\@empty
  \let\@evenfoot\@oddfoot
}
\makeatother
\begin{document}
\begin{frontmatter}
\title{Super-Exponential Approximation of the Riemann-Liouville Fractional Integral via Gegenbauer-Based Fractional Approximation Methods}
\author[Ajman,NDRC]{Kareem T. Elgindy\corref{cor1}}
\ead{k.elgindy@ajman.ac.ae}
\address[Ajman]{Department of Mathematics and Sciences, College of Humanities and Sciences, Ajman University, P.O.Box: 346 Ajman, United Arab Emirates}
\address[NDRC]{Nonlinear Dynamics Research Center (NDRC), Ajman University, P.O.Box: 346 Ajman, United Arab Emirates}
\cortext[cor1]{Corresponding author}

\begin{abstract}
This paper introduces a Gegenbauer-based fractional approximation (GBFA) method for high-precision approximation of the left Riemann-Liouville fractional integral (RLFI). By using precomputable fractional-order shifted Gegenbauer integration matrices (FSGIMs), the method achieves super-exponential convergence for smooth functions, delivering near machine-precision accuracy with minimal computational cost. Tunable shifted Gegenbauer (SG) parameters enable flexible optimization across diverse problems, while rigorous error analysis verifies a fast reduction in approximation error when appropriate parameter choices are applied. Numerical experiments demonstrate that the GBFA method outperforms MATLAB's \texttt{integral}, MATHEMATICA's \texttt{NIntegrate}, and existing techniques by up to two orders of magnitude in accuracy, with superior efficiency for varying fractional orders $0 < \alpha < 1$. Its adaptability and precision make the GBFA method a transformative tool for fractional calculus, ideal for modeling complex systems with memory and non-local behavior, where understanding underlying structures often benefits from recognizing inherent symmetries or patterns.
\end{abstract}

\begin{keyword}
Riemann-Liouville fractional integral \sep Shifted Gegenbauer polynomials \sep Pseudospectral methods \sep Super-exponential convergence \sep Fractional-order integration matrix.\\
\textbf{MSC 2020 Classification:} 26A33 \sep 41A10 \sep 65D30.
\end{keyword}
\end{frontmatter}

\begin{table*}[ht]\caption{\centering Table of Symbols and Their Meanings}\centering
\resizebox{\textwidth}{!}{%
\begin{tabular}{|c|c|c|c|c|c|}
\hline
\textbf{Symbol} & \textbf{Meaning} & \textbf{Symbol} & \textbf{Meaning} & \textbf{Symbol} & \textbf{Meaning} \\
\hline
$\forall$ & for all & $\foralla$ & for any & $\forallaa$ & for almost all\\
\hline
$\foralle$ & for each & $\foralls$ & for some & $\forallS$ & for (a) relatively small \\
\hline
$\forallL$ & for (a) relatively large & $\ll$ & much less than & $\exists$ & there exist(s) \\
\hline
$\sim$ & asymptotically equivalent & $\simlt$ & asymptotically less than & $\not\approx$ & not sufficiently close to \\
\hline
$\MFC$ & set of all complex-valued functions & $\MFF$ & set of all real-valued functions & $\MBC$ & set of complex numbers \\
\hline
$\MBR$ & set of real numbers & $\MBRczer$ & set of nonzero real numbers & $\MBRzerP$ & set of non-negative real numbers \\
\hline
$\MBRmhzer$ & $\{x \in \MBR: -1/2 < x < 0\}$ & $\MBZ$ & set of integers & $\MBZP$ & set of positive integers \\
\hline
$\MBZzerP$ & set of non-negative integers & $\MBZeP$ & set of positive even integers & $i$:$j$:$k$ & list of numbers from $i$ to $k$ with increment $j$ \\
\hline
$i$:$k$ & list of numbers from $i$ to $k$ with increment 1 & $y_{1:n}$ or $\left. y_i \right|_{i=1:n}$ & list of symbols $y_1, y_2, \ldots, y_n$ & $\{y_{1:n}\}$ & set of symbols $y_1, y_2, \ldots, y_n$ \\
\hline
$\MBJ_n$ & $\{0:n-1\}$ & $\MBJP_n$ & $\MBJ_n \cup \{n\}$ & $\MBN_n$ & $\{1:n\}$ \\
\hline
$\MBN_{m,n}$ & $\{m:n\}$ & $\MBG_n^{\lambda}$ & set of Gegenbauer-Gauss (GG) zeros of the $(n+1)$st-degree Gegenbauer polynomial with index $\lambda > -1/2$ & $\MBSG_n^{\lambda}$ & set of SGG points in the interval $[0, 1]$ \\
\hline
$\FOmega_{a,b}$ & closed interval $[a, b]$ & $\IFOmega$ & interior of the set $\FOmega$ & $\FOmega_{T}$ & specific interval $[0, T]$ \\
\hline
$\FOmega_{L \times T}$ & Cartesian product $\FOmega_{L} \times \FOmega_{T}$ & $\ProductLog$ & Lambert W function & $\Gamma(\cdot)$ & Gamma function \\
\hline
$\Gamma(\cdot,\cdot)$ & upper incomplete gamma function & $\left\lceil {.} \right\rceil$ & ceiling function & $\MBI_{j \geq k}$ & indicator (characteristic) function $\begin{cases}
 1 & \text{if } j \geq k, \\
 0 & \text{otherwise.}\end{cases}$ \\
\hline
$E_{\alpha, \beta}(z)$ & two-parameter Mittag-Leffler function & $(x)_m$ & The generalized falling factorial $\frac{\Gamma(x+1)}{\Gamma(x-m+1)}\,\forall x \in \MBC, m \in \MBZzerP$ & $\supp(f)$ & support of function $f$ \\
\hline
$f^*$ & complex conjugate of $f$ & $f_n$ & $f(t_n)$ & $f_{N,n}$ & $f_N(t_n)$ \\
\hline
$\C{I}_{b}^{(t)}h$ & $\int_0^{b} {h(t)\,dt}$ & $\C{I}_{a, b}^{(t)}h$ & $\int_a^{b} {h(t)\,dt}$ & $\C{I}_t^{(t)}h$ & $\int_0^t {h(.)\,d(.)}$ \\
\hline
$\C{I}_{b}^{(t)}h\cancbra{u(t)}$ & $\int_0^{b} {h(u(t))\,dt}$ & $\C{I}_{a,b}^{(t)}h\cancbra{u(t)}$ & $\int_a^b {h(u(t))\,dt}$ & $\C{I}_{\FOmega_{a,b}}^{(x)} h$ & $\int_a^b {h(x)\,dx}$ \\
\hline
$\RLI{t}{\alpha}{f}$ & The left RLFI defined by $\frac{1}{{\Gamma (\alpha )}}\int_0^t {{{(t - \tau )}^{\alpha  - 1}}} f(\tau )\,d\tau$ & $\Part{x}$ & $d/dx$ & $\NPart{x}{n}$ & $d^n/d x^n$ \\
\hline
$\CapD{x}{\alpha}{f}$ & $\alpha$th-order Caputo fractional derivative of $f$ at $x$ & $\Def{\FOmega}$ & space of all functions defined on $\FOmega$ & $C^k(\FOmega)$ & space of $k$ times continuously differentiable functions on $\FOmega$ \\
\hline
$L^p({\FOmega})$ & Banach space of measurable functions $u \in \Def{\FOmega}$ with ${\left\| u \right\|_{{L^p}}} = {\left( {{\C{I}_{\FOmega}}{{\left| u \right|}^p}} \right)^{1/p}} < \infty$ & $L^{\infty}({\FOmega})$ & space of all essentially bounded measurable functions on $\FOmega$ & $\left\|f\right\|_{L^{\infty}(\FOmega)}$ & $L^{\infty}$ norm: $\sup_{x \in \FOmega} |f(x)| = \inf\{M \ge 0: |f(x)| \le M\,\forallaa x \in \FOmega\}$ \\
\hline
$\left\|\cdot\right\|_1$ & $l_1$-norm & $\left\|\cdot\right\|_2$ & Euclidean norm & $\ME{H}^{k,p}(\FOmega)$ & Sobolev space of weakly differentiable functions with integrable weak derivatives up to order $k$ \\
\hline
$\bmt_N$ & $[t_{N,0}, t_{N,1}, \ldots, t_{N,N}]^{\top}$ & $g_{0:N}$ & $[g_0, g_1, \ldots, g_{N}]^{\top}$ & $g^{(0:N)}$ & $[g, g', \ldots, g^{(N)}]^{\top}$ \\
\hline
$c^{0:N}$ & $[1, c, c^2, \ldots, c^{N}]$ & $\bmt_N^{\top}$ or $[t_{N,0:N}]$ & $[t_{N,0}, t_{N,1}, \ldots, t_{N,N}]$ & $h(\bmy)$ & vector with $i$-th element $h(y_i)$ \\
\hline
$\bmh(\bmy)$ or $h_{1:m}\cancbra{\bmy}$ & $[h_1(\bmy), \ldots, h_m(\bmy)]^{\top}$ & $\bmy^{\div}$ & vector of reciprocals of the elements of $\bmy$ & $\F{O}_n$ & zero matrix of size $n$ \\
\hline
$f(n) = O(g(n))$ & $\exists\,n_0, c > 0: 0 \le f(n) \le c g(n)\,\forall n \ge n_0$ & $f(n) = o(g(n))$ & $\lim_{n\to \infty} \frac{f(n)}{g(n)} = 0$ && \\
\hline
\end{tabular}
}%
\label{tab:symbols}
\caption*{\textit{Remark: A vector is represented in print by a bold italicized symbol while a two-dimensional matrix is represented by a bold symbol, except for a row vector whose elements form a certain row of a matrix where we represent it in bold symbol.}}
\end{table*}
\section{Introduction}
\label{Int}
Fractional calculus offers a powerful framework for modeling intricate systems characterized by memory and non-local interactions, finding applications in diverse fields such as viscoelasticity \cite{mainardi2022fractional}, anomalous diffusion \cite{gorenflo1998random}, and control theory \cite{monje2010fractional}, among others. A fundamental concept in this domain is the left RLFI, defined for $\alpha \in (0,1)$ and $f \in L^2(\Omega_1)$ as detailed in Table \ref{tab:symbols}. In contrast to classical calculus, which operates under the assumption of local dependencies, fractional integrals, exemplified by the left RLFI, inherently account for cumulative effects over time through a singular kernel. This characteristic renders them particularly well-suited for modeling phenomena where past states significantly influence future behavior. The RLFI proves especially valuable in the description of complex dynamics exhibiting self-similarity, scale-invariance, or memory effects, often unveiling underlying symmetries that facilitate analysis and prediction. Nevertheless, the singular nature of the RLFI's kernel, $(t-\tau)^{\alpha-1}$, presents substantial computational hurdles, as conventional numerical methods frequently encounter limitations in accuracy or incur high computational costs. Consequently, the development of efficient and high-precision approximation techniques for the RLFI is of paramount importance for advancing computational modeling across physics, engineering, and biology, where fractional calculus is increasingly employed to tackle real-world problems involving non-local behavior or fractal structures. 


\begin{table*}[ht]\caption{\centering Table of Symbols and Their Meanings}\centering
\resizebox{0.8\textwidth}{!}{%
\begin{tabularx}{\textwidth}{|X|}
\hline
\multicolumn{1}{|c|}{\textbf{Logical Operators and Quantifiers}} \\
\hline \vspace{-1mm}
\begin{center}
 $\forall$ : for all \qquad $\foralla$ : for any \qquad $\forallaa$ : for almost all \qquad $\foralle$ : for each \qquad $\foralls$ : for some \qquad $\forallS$ : for (a) relatively small \qquad $\forallL$ : for (a) relatively large \qquad $\exists$ : there exist(s) 
 \end{center} \\[1.5em]
\hline
\multicolumn{1}{|c|}{\textbf{Comparison and Relation Symbols}} \\
\hline\\[-1.5em]
\begin{center}
 $\ll$ : much less than \qquad $\sim$ : asymptotically equivalent \qquad $\simlt$ : asymptotically less than \qquad $\not\approx$ : not sufficiently close to\\
$f(n) = O(g(n))$ : $\exists\,n_0, c > 0: 0 \le f(n) \le c g(n)\,\forall n \ge n_0$ \qquad $f(n) = o(g(n))$ : $\lim_{n\to \infty} \frac{f(n)}{g(n)} = 0$
 \end{center} \\[2.5em]
\hline
\multicolumn{1}{|c|}{\textbf{Sets and Number Systems}} \\
\hline\\[-1.5em]
\begin{center}
 $\MBC$ : set of complex numbers \qquad $\MBR$ : set of real numbers \qquad $\MBR_\Theta$ : set of nonzero real numbers \qquad $\MBRzerP$ : set of non-negative real numbers \qquad $\MBRmhzer$ : $\{x \in \MBR: -1/2 < x < 0\}$ \qquad $\MBZ$ : set of integers\\
 $\MBZP$ : set of positive integers \qquad $\MBZzerP$ : set of non-negative integers \qquad $\MBZeP$ : set of positive even integers \qquad $\MBJ_n$ : $\{0:n-1\}$ \qquad $\MBJP_n$ : $\MBJ_n \cup \{n\}$ \qquad $\MBN_n$ : $\{1:n\}$ \qquad $\MBN_{m,n}$ : $\{m:n\}$ \qquad $\{y_{1:n}\}$ : set of symbols $y_1, y_2, \ldots, y_n$ \\
$\MBG_n^{\lambda}$ : set of Gegenbauer-Gauss zeros of the $(n+1)$st-degree Gegenbauer polynomial with index $\lambda > -1/2$ \qquad $\MBSG_n^{\lambda}$ : set of SGG points in the interval $[0, 1]$ \\
$\bm{\Omega}_{a,b}$ : closed interval $[a, b]$ \qquad ${\bm{\Omega}}$\textsuperscript{◦} : interior of the set $\bm{\Omega}$ \qquad $\bm{\Omega}_{T}$ : specific interval $[0, T]$ \qquad $\bm{\Omega}_{L \times T}$ : Cartesian product $\bm{\Omega}_{L} \times \bm{\Omega}_{T}$\\
$N_\delta(a) = \{x \mid d(x, a) < \delta\}: 0 < \delta \ll 1$ and $d(x,a)$ is the distance (or metric) between the point $x$ and the point $a$.
 \end{center} \\[4em]
\hline
\multicolumn{1}{|c|}{\textbf{Lists and Sequences}} \\
\hline
\begin{center}
 $i$:$j$:$k$ : list of numbers from $i$ to $k$ with increment $j$ \qquad $i$:$k$ : list of numbers from $i$ to $k$ with increment 1 \qquad $y_{1:n}$ or $\left. y_i \right|_{i=1:n}$ : list of symbols $y_1, y_2, \ldots, y_n$
 \end{center} \\[1.5em]
\hline
\multicolumn{1}{|c|}{\textbf{Functions and Special Functions}} \\
\hline\\[-1.25em]
\begin{center}
 $\ProductLog$ : Lambert W function \qquad $\Gamma(\cdot)$ : Gamma function \qquad $\Gamma(\cdot,\cdot)$ : upper incomplete gamma function \qquad $\left\lceil {.} \right\rceil$ : ceiling function \qquad $E_{\alpha, \beta}(z)$ : two-parameter Mittag-Leffler function \\
$_1F_1(a;c;z)$: The confluent hypergeometric function defined as $\sum\limits_{n=0}^\infty \frac{a^{\underline{n}}}{c^{\underline{n}}}\frac{z^n}{n!}$ where $q^{\underline{n}}$ is the Pochhammer symbol\\
$_2F_1(a,b,c,z)$: The Gauss hypergeometric function defined as $\sum\limits_{n=0}^\infty \frac{a^{\underline{n}}\,b^{\underline{n}}}{c^{\underline{n}}}\frac{z^n}{n!}$\\
$G_n^{\lambda}$: $n$th-degree Gegenbauer polynomial with index $\lambda > -1/2$\qquad $\hG_n^{\lambda}$: $n$th-degree SG polynomial with index $\lambda > -1/2$ defined on $\bm{\Omega}_1$\\
$(x)_m$ : The generalized falling factorial $\frac{\Gamma(x+1)}{\Gamma(x-m+1)}\,\forall x \in \MBC, m \in \MBZzerP$ \qquad $\delta_{m,n}$: The Kronecker delta with integer indices $m$ and $n$ \qquad $\supp(f)$ : support of function $f$ \qquad $f^*$ : complex conjugate of $f$\\
$\MBI_{j \geq k}$ : indicator (characteristic) function $\begin{cases}
1 & \text{if } j \geq k, \\
0 & \text{otherwise.}\end{cases}$
 \end{center}
 \text{}\\[7em]
\hline
\multicolumn{1}{|c|}{\textbf{Notation and Shorthands}} \\
\hline \vspace{-2mm}
\begin{center}
$f_n$ : $f(t_n)$ \quad $f_{N,n}$ : $f_N(t_n)$
\end{center} \\[1em]
\hline
\multicolumn{1}{|c|}{\textbf{Function Spaces}} \\
\hline\\[-1.5em]
\begin{center}
 $\MFC$ : set of all complex-valued functions \qquad $\MFF$ : set of all real-valued functions \qquad $\Def{\bm{\Omega}}$ : space of all functions defined on $\bm{\Omega}$ \qquad $C^k(\bm{\Omega})$ : space of $k$ times continuously differentiable functions on $\bm{\Omega}$ \\
$L^p({\bm{\Omega}})$ : Banach space of measurable functions $u \in \Def{\bm{\Omega}}$ with ${\left\| u \right\|_{{L^p}}} = {\left( {{\C{I}_{\bm{\Omega}}}{{\left| u \right|}^p}} \right)^{1/p}} < \infty$ \qquad $L^{\infty}({\bm{\Omega}})$ : space of all essentially bounded measurable functions on $\bm{\Omega}$ \\
$\ME{H}^{k,p}(\bm{\Omega})$ : Sobolev space of weakly differentiable functions with integrable weak derivatives up to order $k$
 \end{center} \\[3em]
\hline
\multicolumn{1}{|c|}{\textbf{Integrals and Derivatives}} \\
\hline
\begin{center}
 $\C{I}_{b}^{(t)}h$ : $\int_0^{b} {h(t)\,dt}$ \qquad $\C{I}_{a, b}^{(t)}h$ : $\int_a^{b} {h(t)\,dt}$ \qquad $\C{I}_t^{(t)}h$ : $\int_0^t {h(.)\,d(.)}$ \qquad $\C{I}_{b}^{(t)}h[u(t)]$ : $\int_0^{b} {h(u(t))\,dt}$ \qquad $\C{I}_{a,b}^{(t)}h[u(t)]$ : $\int_a^b {h(u(t))\,dt}$ \qquad $\C{I}_{\bm{\Omega}_{a,b}}^{(x)} h$ : $\int_a^b {h(x)\,dx}$ \qquad $\RLI{t}{\alpha}{f}$ : The left RLFI defined by $\frac{1}{{\Gamma (\alpha )}}\int_0^t {{{(t - \tau )}^{\alpha  - 1}}} f(\tau )\,d\tau$ \\
$\Part{x}$ : $d/dx$ \qquad $\NPart{x}{n}$ : $d^n/d x^n$ \qquad $\CapD{x}{\alpha}{f}$ : $\alpha$th-order Caputo fractional derivative of $f$ at $x$
 \end{center} \\[3em]
\hline
\multicolumn{1}{|c|}{\textbf{Norms and Metrics}} \\
\hline
\begin{center}
 $\left\|f\right\|_{L^{\infty}(\bm{\Omega})}$ : $L^{\infty}$ norm: $\sup_{x \in \bm{\Omega}} |f(x)| = \inf\{M \ge 0: |f(x)| \le M\,\forallaa x \in \bm{\Omega}\}$ \qquad $\left\|\cdot\right\|_1$ : $l_1$-norm \qquad $\left\|\cdot\right\|_2$ : Euclidean norm
 \end{center} \\[2em]
\hline
\multicolumn{1}{|c|}{\textbf{Vectors and Matrices}} \\
\hline
\begin{center}
 $\bmt_N$ : $[t_{N,0}, t_{N,1}, \ldots, t_{N,N}]^{\top}$ \qquad $g_{0:N}$ : $[g_0, g_1, \ldots, g_{N}]^{\top}$ \qquad $g^{(0:N)}$ : $[g, g', \ldots, g^{(N)}]^{\top}$ \qquad $c^{0:N}$ : $[1, c, c^2, \ldots, c^{N}]$ \qquad $\bmt_N^{\top}$ or $[t_{N,0:N}]$ : $[t_{N,0}, t_{N,1}, \ldots, t_{N,N}]$ \qquad $h(\bmy)$ : vector with $i$-th element $h(y_i)$ \\
$\bmh(\bmy)$ or $h_{1:m}[\bmy]$ : $[h_1(\bmy), \ldots, h_m(\bmy)]^{\top}$ \qquad $\bmy^{\div}$ : vector of reciprocals of the elements of $\bmy$ \qquad $\F{O}_n$ : zero matrix of size $n$
 \end{center} \\[3em]
\hline
\end{tabularx}
}
\label{tab:symbols}
\caption*{\textit{Remark: A vector is represented in print by a bold italicized symbol while a two-dimensional matrix is represented by a bold symbol, except for a row vector whose elements form a certain row of a matrix where we represent it in bold symbol.}}
\end{table*}

Existing approaches to RLFI approximation include wavelet-based methods \cite{Zhang2025,Ghasempour2025753,Rabiei2021221,rahimkhani2025numerical,damircheli2024wavelet,rahimkhani2017numerical}, polynomial and orthogonal function techniques \cite{barary2024efficient,deniz2023numerical,postavaru2023efficient,akhlaghi2023application,bazgir2020existence}, finite difference and quadrature schemes \cite{cao2023optimal,qiu2021crank,cui2023alternating,diethelm2005algorithms}, and operational matrix methods \cite{edrisi2022using,avci2020numerical,xiaogang2018operational,krishnasamy2017numerical}. Additional methods involve local meshless techniques \cite{nikan2021numerical}, alternating direction implicit schemes \cite{zhai2016investigations}, and radial basis functions \cite{thakoor2023new}. While these methods have shown promise in specific contexts, they often struggle with trade-offs between accuracy, computational cost, and flexibility, particularly when adapting to diverse problem characteristics or varying fractional orders.

This study introduces the GBFA method, which overcomes these challenges through three key innovations: (i) \textit{Parameter adaptability}, utilizing the tunable SG parameters $\lambda$ (for interpolation) and $\lambda_q$ (for quadrature approximation) to optimize performance across a wide range of problems; (ii) \textit{Super-exponential convergence}, achieving rapid error decay for smooth functions, often reaching near machine precision with modest node counts; and (iii) \textit{Computational efficiency}, enabled by precomputable FSGIMs that minimize runtime costs. The proposed GBFA method addresses these challenges by using the orthogonality and flexibility of SG polynomials to achieve super-exponential convergence. This approach offers near machine-precision accuracy with minimal computational effort, particularly for systems necessitating repeated fractional integrations or displaying symmetric patterns. Numerical experiments demonstrate that the GBFA method significantly outperforms established tools like MATLAB's \texttt{integral} function and MATHEMATICA's \texttt{NIntegrate}, achieving up to two orders of magnitude higher accuracy in certain cases. It also surpasses prior methods, such as the trapezoidal approach of \citet{dimitrov2021approximations}, the spline-based techniques of \citet{ciesielski2024numerical}, and the neural network method \cite{nowak2025neural}, in both precision and efficiency. Sensitivity analysis reveals that setting $\lambda_q < \lambda$ often accelerates convergence, while rigorous error bounds confirm super-exponential decay under optimal parameter choices. The FSGIM's invariance for fixed points and parameters enables precomputation, making the GBFA method ideal for problems requiring repeated fractional integrations with varying $\alpha$. While the method itself exploits the mathematical properties of orthogonal polynomials, its application can be crucial in analyzing systems where symmetry or repeating patterns are fundamental, as fractional calculus is used to model phenomena with memory where past states influence future behavior, and identifying symmetries in such systems can simplify analysis and prediction.

The paper is organized as follows: Section \ref{sec:RLFI} presents the GBFA framework. Section \ref{sec:comp_complexity} analyzes computational complexity. Section \ref{sec:error_analysis} provides a detailed error analysis. Section \ref{subsec:PGFPS1} provides actionable guidelines for selecting the tunable parameters $\lambda$ and $\lambda_q$, balancing accuracy and computational efficiency. Section \ref{sec:FNS} evaluates numerical performance. Section \ref{sec:Conc} presents the conclusions of this study with future works of potential methodological extensions to non-smooth functions encountered in fractional calculus. Finally, \ref{sec:App1} includes supporting mathematical proofs.

\section{Numerical Approximation of RLFI}
\label{sec:RLFI}
This section presents the GBFA method adapted for approximating the RLFI. We shall briefly begin with some background on the classical Gegenbauer polynomials, as they form the foundation for the GBFA method developed in this study.

The classical Gegenbauer polynomials \( G_n^\lambda(x) \), defined on \( x \in [-1, 1] \) for \(\lambda > -1/2\), are orthogonal with respect to the weight function \(w^{\lambda}(x) = (1-x^2)^{\lambda-1/2} \). They satisfy the three-term recurrence relation 
\begin{equation}
(n+2 \alpha) G_{n+1}^\lambda(x) = 2(n+\lambda) x G_n^\lambda(x) - n G_{n-1}^\lambda(x), 
\end{equation}
with \( G_0^\lambda(x) = 1 \) and \( G_1^\lambda(x) = x \). Their orthogonality is given by 
\begin{equation}
\C{I}_{-1,1}^{(t)} {\left(w^{\lambda} G_n^\lambda G_m^\lambda\right)} = h_n^\lambda\,\delta_{m,n},
\end{equation}
 where 
\begin{equation}
h_n^\lambda = \frac{2^{1-2\lambda}\,\pi\,\Gamma(n+2\lambda)}{n!\,(n+\lambda)\,\Gamma^2(\lambda)}. 
\end{equation} 
The SG polynomials \(\hat{G}_n^\lambda(t)\), used in this study, are obtained via the transformation \( x = 2t - 1 \) for \( t \in \bm{\Omega}_1 \), inheriting analogous properties adjusted for the shifted domain. For a comprehensive treatment of their properties and quadrature rules, see \cite{Elgindy20161,Elgindy20171,elgindy2018optimal,elgindy2018high}.

Let \(\alpha \in (0,1)$, $f \in L^2(\bm{\Omega}_1)$, and $\left\{\hat{t}_{n,0:n}^{\lambda}\right\} = \MBSG_n^{\lambda}$. The GBFA interpolant of $f$ is given by
\begin{equation}\label{eq:RL_Lagint}
	I_n f(t) = f_{0:n}^{\top} \C{L}_{0:n}^{\lambda}[t],
\end{equation}
where $\C{L}_{k}^{\lambda}(t)$ is defined as
\begin{equation}\label{eq:RL_Lag}
	\C{L}_k^{\lambda}(t) = \hvarpi_k^{\lambda} \trp{{\lambdabar_{0:n}^{\lambda}}^{\div}} \left( \hG_{0:n}^{\lambda}[\hat{t}_{n,k}^{\lambda}] \odot \hG_{0:n}^{\lambda}[t] \right), \quad \forall k \in \MBJP_n,
\end{equation}
with normalization factors and Christoffel numbers:
\begin{gather}
	\hlambdabar_j^{\lambda} = \frac{\pi 2^{1 - 4\lambda} \Gamma(j + 2\lambda)}{j! \Gamma^2(\lambda) (j + \lambda)}, \\
	\hvarpi_k^{\lambda} = 1 / \left[ \trp{\hlambdabar_{0:n}^{\lambda}}^{\div} \left( \hG_{0:n}^{\lambda}[\hat{t}_{n,k}^{\lambda}] \right)_{(2)} \right],
\end{gather}
$\forall j, k \in \MBJP_n$. In matrix form, we can write Eq. \eqref{eq:RL_Lag} as
\begin{equation}\label{eq:RL_matf}
	\C{L}_{0:n}^{\lambda}[t] = \Diag{\hvarpi_{0:n}^{\lambda}} \left( \hG_{0:n}^{\lambda}[t \bm{\mathit{1}}_{n+1}] \odot \hG_{0:n}^{\lambda}[\hbmt_n^{\lambda}] \right)^{\top} {\hlambdabar_{0:n}^{\lambda}} {}^{\div}.
\end{equation}
This allows us to approximate the RLFI as follows:
\begin{equation}\label{eq:RL_approx}
	\RLI{t}{\alpha}{f} \approx \RLI{t}{\alpha}{I_n f} = f_{0:n}^{\top} \RLI{t}{\alpha}{\C{L}_{0:n}^{\lambda}}.
\end{equation}
Using the transformation 
\begin{equation}\label{eq:Trans111}
\tau = t \left(1 - y^{1/\alpha}\right),
\end{equation}
Formula \eqref{eq:RL_approx} becomes
\begin{equation}\label{eq:RL_transform}
	\RLI{t}{\alpha}{f} \approx \frac{t^\alpha}{\Gamma(\alpha + 1)} f_{0:n}^{\top} \C{I}_1^{(y)} \C{L}_{0:n}^{\lambda}[t (1 - y^{1/\alpha})].
\end{equation}
Substituting Eq. \eqref{eq:RL_matf} into Eq. \eqref{eq:RL_transform}:
\begin{equation}
\begin{array}{cc}
	{\RLI{t}{\alpha}{f} \approx \frac{t^\alpha}{\Gamma(\alpha + 1)} \left[\trp{{\hlambdabar_{0:n}^{\lambda}} {}^{\div}} \times\notag\right.}\\
	{\left.\left( \C{I}_1^{(y)} \hG_{0:n}^{\lambda}[t (1 - y^{1/\alpha}) \bmone_{n+1}] \odot \hG_{0:n}^{\lambda}[\hbmt_n^{\lambda}] \right) \Diag{\hvarpi_{0:n}^{\lambda}}\right] f_{0:n}.}
	\label{eq:RL_single}
\end{array}
\end{equation}

For multiple points $z_{0:M} \in \bm{\Omega}_1\,\foralls M \in \MBZzerP$, we extend Eq. \eqref{eq:RL_matf}:
\begin{equation}
\begin{array}{ccc}
	{\C{L}_{0:n}^{\lambda}[\bmz_M] = \text{resh}_{n+1,M+1} \left[\trp{{\hlambdabar_{0:n}^{\lambda}} {}^{\div}} \times\notag\right.}\\
	{\left. \left( \hG_{0:n}^{\lambda}[\bmz_M \otimes \bmone_{n+1}] \odot \hG_{0:n}^{\lambda}[\bmone_{M+1} \otimes \hbmt_n^{\lambda}] \right) \times\notag\right.}\\
	{\left. \left( \F{I}_{M+1} \otimes \Diag{\hvarpi_{0:n}^{\lambda}} \right)\right].} \label{eq:RL_Lagmat}
	\end{array}
\end{equation}
Thus,
\begin{equation}\label{eq:RL_multi1}
	\RLI{\bmz_M}{\alpha}{f} \approx \frac{1}{\Gamma(\alpha + 1)} \left[ \EE{\bmz_M}{\alpha} \odot \left( \hRLIM{n}{\alpha} f_{0:n} \right) \right],
\end{equation}
where
\begin{equation}
\begin{array}{ccc}
	{\hRLIM{n}{\alpha} = \text{resh}_{n+1,M+1}^{\top}\left[\trp{{\hlambdabar_{0:n}^{\lambda}} {}^{\div}} \times\notag\right.}\\
	{\left.\left( \C{I}_1^{(y)} \hG_{0:n}^{\lambda}[\bmz_M \otimes (1 - y^{1/\alpha}) \bmone_{n+1}] \odot \hG_{0:n}^{\lambda}[\bmone_{M+1} \otimes \hbmt_n^{\lambda}] \right) \times\notag\right.}\\
	{\left. \left( \F{I}_{M+1} \otimes \Diag{\hvarpi_{0:n}^{\lambda}} \right)\right].} \label{eq:RL_hRLIM}
\end{array}
\end{equation}
Alternatively,
\begin{equation}\label{eq:RL_multi2}
	\RLI{\bmz_M}{\alpha}{f} \approx \RLIM{n}{\alpha} f_{0:n},
\end{equation}
where
\begin{equation}\label{eq:RL_RLIM}
	\RLIM{n}{\alpha} = \frac{1}{\Gamma(\alpha + 1)} \Diag{\EE{\bmz_M}{\alpha}} \hRLIM{n}{\alpha}.
\end{equation}
We term $\RLIM{n}{\alpha}$ the ``$\alpha$th-order FSGIM'' for the RLFI and $\hRLIM{n}{\alpha}$ the ``$\alpha$th-order FSGIM Generator.'' Eq. \eqref{eq:RL_multi1} is preferred for computational efficiency.

To compute $\C{I}_1^{(y)} \hG_j^{\lambda}[t (1 - y^{1/\alpha})]$, we can use the SGIRV $\hFP = \frac{1}{2} \F{P}$ with SGG nodes $\hbmt_{n_q}^{\lambda_q}$:
\begin{equation}\label{eq:RL_quad}
	\C{I}_1^{(y)} \hG_j^{\lambda}[t (1 - y^{1/\alpha})] \approx \hFP \hG_j^{\lambda}\left( t \left( 1 - {\hbmt_{n_q}^{\EE{\lambda_q}{1/\alpha}}} \right) \right), \quad \forall j \in \MBJP_n, \, t \in \bm{\Omega}_1,
\end{equation}
cf. \cite[Algorithm 6 or 7]{Elgindy20171}. Formula \eqref{eq:RL_quad} represents the $(n_q, \lambda_q)$-GBFA quadrature used for the numerical partial calculation of the RLFI. We denote the approximate $\alpha$th-order RLFI of a function at point $t$, computed using Eq. \eqref{eq:RL_quad} in conjunction with either Eq. \eqref{eq:RL_multi1} or Eq. \eqref{eq:RL_RLIM}, by $\RLIEN{t}{\alpha}{}$.

Figure \ref{fig:GBFA_workflow} provides a visual summary of the GBFA method's workflow, illustrating the seamless integration of interpolation, transformation, and quadrature steps. This schematic highlights the method's flexibility, as the tunable parameters $\lambda$ and $\lambda_q$ allow practitioners to tailor the approximation to specific problem characteristics, optimizing both accuracy and computational efficiency. The alternative path of precomputing the FSGIM, as indicated by the dashed arrow, underscores the method's suitability for applications requiring repeated evaluations.

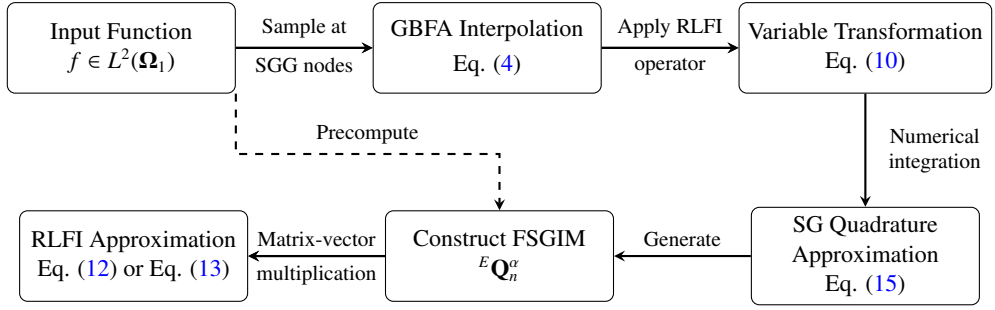
\begin{figure*}[ht]
\centering
\begin{tikzpicture}[
  node distance=1.5cm and 1.8cm,
  box/.style={
    draw,
    rectangle,
    rounded corners=3pt,
    minimum width=3cm,
    minimum height=1.2cm,
    align=center,
    font=\small,
    fill=white
  },
  arrow/.style={
    ->,
    >=stealth,
    thick
  },
  dashedarrow/.style={
    ->,
    >=stealth,
    thick,
    dashed
  }
]
\node[box] (input) {Input Function\\$f \in L^2(\bm{\Omega}_1)$};
\node[box, right=of input] (interp) {GBFA Interpolation\\[0.25em]Eq. \eqref{eq:RL_Lagint}};
\node[box, right=of interp] (transform) {Variable Transformation\\Eq. \eqref{eq:Trans111}};
\node[box, below=of transform] (quad) {SG Quadrature\\Approximation\\Eq. \eqref{eq:RL_quad}};
\node[box, left=of quad] (fsgim) {Construct FSGIM\\$^E\mathbf{Q}_n^\alpha$};
\node[box, left=of fsgim] (output) {RLFI Approximation\\Eq. \eqref{eq:RL_multi1} or Eq. \eqref{eq:RL_multi2}};

\draw[arrow] (input) -- node[above] {\footnotesize Sample at} (interp);
\draw[arrow] (input) -- node[below] {\footnotesize SGG nodes} (interp);
\draw[arrow] (interp) -- node[above] {\footnotesize Apply RLFI} (transform);
\draw[arrow] (interp) -- node[below] {\footnotesize operator} (transform);
\draw[arrow] (transform) -- node[right] {\footnotesize $\begin{array}{*{20}{c}}
{{\text{Numerical}}}\\
{{\text{integration}}}
\end{array}$} (quad);
\draw[arrow] (quad) -- node[above] {\footnotesize Generate} (fsgim);
\draw[arrow] (fsgim) -- node[above] {\footnotesize Matrix-vector} (output);
\draw[arrow] (fsgim) -- node[below] {\footnotesize multiplication} (output);

\draw[dashedarrow] (input.south east) |- ++(0,-0.8cm) -| node[pos=0.25, above] {\footnotesize Precompute} (fsgim.north);
\end{tikzpicture}
\caption{Workflow of the GBFA method for RLFI approximation. The main path (solid arrows) shows the standard procedure: (1) interpolate the input function at SGG nodes, (2) apply the RLFI operator with variable transformation, (3) approximate the integrals of SG polynomials using SG quadrature, (4) construct the FSGIM, and (5) compute the final approximation. The dashed arrow indicates an alternative path: precomputing the FSGIM for direct evaluation and repeated use. The tunable parameters $\lambda$ (interpolation) and $\lambda_q$ (quadrature) enable optimization across different problems.}
\label{fig:GBFA_workflow}
\end{figure*}

\section{Computational Complexity}
\label{sec:comp_complexity}

This section provides a computational complexity analysis of constructing the \(\alpha$th-order FSGIM \(\RLIM{n}{\alpha}$ and its generator \(\hRLIM{n}{\alpha}$. The analysis is based on the key matrix operations involved in the construction process, which we analyze individually in the following:

\begin{itemize}
\item The term $\EE{\bmz_M}{\alpha}$ involves raising each element of an $(M+1)$-dimensional vector to the power $\alpha$. This operation requires $O(M)$ operations.
\item Constructing $\RLIM{n}{\alpha}$ from $\hRLIM{n}{\alpha}$ involves a diagonal scaling by $\Diag{\EE{\bmz_M}{\alpha}}$, which requires another $O(M n)$ operations.
\item The matrix $\hRLIM{n}{\alpha}$ is constructed using several matrix multiplications and elementwise operations. For each entry of $\bmz_M$, the dominant steps include:
\begin{itemize}
\item The computation of $\hG_{0:n}^{\lambda}$ using the three-term recurrence relation requires $O(n)$ operations per point. Since the polynomial evaluation is required for polynomials up to degree $n$, this requires $O(n^2)$ operations.
\item The quadrature approximation involves evaluating a polynomial at transformed nodes. The cost of calculating $\hFP$ depends on the chosen methods for computing factorials and the Gamma function, which can be considered a constant overhead. The computation of ${\hbmt_{n_q}^{\EE{\lambda_q}{1/\alpha}}}$ involves raising each element of the column vector $\hbmt_{n_q}^{\lambda_q}$ to the power $1/\alpha$, which is linear in $(n_q + 1)$. The cost of the matrix-vector multiplication is also linear in $n_q + 1$. Therefore, the computational cost of this step is $O(n_q)$ for each $j \in \MBJP_n$. The overall cost, considering all polynomial functions involved, is $O(n n_q)$.
\item The Hadamard product introduces another $O(n^2)$ operations.
\item The evaluation of $\left({\hlambdabar_{0:n}^{\lambda}} {}^{\div}\right)$ requires $O(n)$ operations, and the product of $\left({\hlambdabar_{0:n}^{\lambda}} {}^{\div}\right)$ by the result from the Hadamard product requires $O(n^2)$ operations.
\item The final diagonal scaling by $\Diag{\hvarpi_{0:n}^{\lambda}}$ contributes $O(n)$.
\end{itemize}   
\end{itemize}
Summing the dominant terms, the overall computational complexity of constructing $\hRLIM{n}{\alpha}$ is $O(n(n + n_q))$ per entry of $\bmz_M$. Therefore, the total number of operations required to construct the matrix $\RLIM{n}{\alpha}$ for all entries of $\bmz_M$ is $O(Mn(n + n_q))$.

Once the FSGIM is precomputed, applying it to compute the RLFI of a function requires only a matrix-vector multiplication with complexity $O(M n)$. The FSGIM's invariance for fixed points and parameters enables precomputation, making the GBFA method ideal for problems requiring repeated fractional integrations with varying $\alpha$. This indicates that the method is particularly efficient when: (i) Multiple integrations are needed with the same parameters, (ii) the same set of evaluation points is used repeatedly, and (iii) different functions need to be integrated with the same fractional order. The precomputation approach becomes increasingly advantageous as the number of repeated evaluations increases, since the one-time $O(Mn(n + n_q))$ cost is amortized across multiple $O(M n)$ applications.

\section{Error Analysis}
\label{sec:error_analysis}

The following theorem establishes the truncation error of the $\alpha$th-order GBFA quadrature associated with the $\alpha$th-order FSGIM $\RLIM{n}{\alpha}$ in closed form.

\begin{Theorem}\label{eq:gsadffhj1}
Suppose that $f \in C^{n+1}(\bm{\Omega}_1)$ is approximated by the GBFA interpolant \eqref{eq:RL_Lagint}. Assume also that the integrals
\begin{equation}
    \C{I}_1^{(y)} \hG_{0:n}^{\lambda}[t (1 - y^{1/\alpha})],
\end{equation}
are computed exactly $\foralla t \in \bm{\Omega}_1$. Then $\exists\,\xi = \xi(t) \in \bm{\Omega}_1^{\circ}$ such that the truncation error, ${}^{\alpha}\ME{T}_n^{\lambda}(t, \xi)$, in the RLFI approximation \eqref{eq:RL_single} is given by
\begin{equation}\label{eq:eeer1}
    {}^{\alpha}\ME{T}_n^{\lambda}(t, \xi) = \frac{t^{\alpha} f^{(n+1)}(\xi)}{(n+1)! \Gamma(\alpha+1) \hK_{n+1}^{\lambda}} \C{I}_1^{(y)} \hG_{n+1}^{\lambda}[t (1 - y^{1/\alpha})].
\end{equation}
\end{Theorem}

\begin{proof}
The Lagrange interpolation error associated with the GBFA interpolation \eqref{eq:RL_Lagint} is given by
\begin{equation}\label{eq:Mar2420251}
    f(t) = I_n f(t) + \frac{f^{(n+1)}(\xi)}{(n+1)! \hK_{n+1}^{\lambda}} \hG_{n+1}^{\lambda}(t),
\end{equation}
where $\hK_{n+1}^{\lambda}$ is the leading coefficient of the $(n+1)$st-degree, $\lambda$-indexed SG polynomial. Applying the RLFI operator $\RLI{t}{\alpha}{}$ on both sides of Eq. \eqref{eq:Mar2420251} results in the truncation error
\begin{equation}\label{eq:fgHH1}
    {}^{\alpha}\ME{T}_n^{\lambda}(t, \xi) = \frac{f^{(n+1)}(\xi)}{(n+1)! \hK_{n+1}^{\lambda}} \RLI{t}{\alpha}{\hG_{n+1}^{\lambda}}.
\end{equation}
The proof is accomplished from Eq. \eqref{eq:fgHH1} by applying the change of variables \eqref{eq:Trans111} on the RLFI of $\hG_{n+1}^{\lambda}$.
\end{proof}

The following theorem provides an upper bound for the truncation error \eqref{eq:eeer1}.

\begin{Theorem}\label{thm:kl1}
Let $\| f^{(n)} \|_{L^{\infty}(\bm{\Omega}_1)} = \C{A}_n$, and suppose that the assumptions of Theorem 4.1 hold. Then the truncation error ${}^{\alpha}\ME{T}_n^{\lambda}(t, \xi)$ satisfies the asymptotic bound
\begin{gather}
    \left| {}^{\alpha}\ME{T}_n^{\lambda}(t, \xi) \right| \simlt \C{A}_{n+1}\,\vartheta_{\alpha, \lambda} {\left( {\frac{e}{4}} \right)^n}{n^{ - \frac{3}{2} - n + \lambda }}\,\Upsilon_{\sigma^{\lambda}}(n)\quad \forallL n,\label{eq:asymptineqsd1}
\end{gather}
where 
\[\Upsilon_{\sigma^{\lambda}}(n) = \left\{ \begin{array}{l}
1,\quad \lambda  \in \MBRzerP,\\
\sigma^{\lambda} n^{-\lambda},\quad \lambda \in \MBRmhzer,
\end{array} \right.
\]
$\sigma^{\lambda} > 1$ is a constant dependent on $\lambda$, and
\begin{gather}
\vartheta_{\alpha,\lambda} = \tfrac{1}{{4\pi }} {e^{\alpha  + 1}} {\alpha ^{ - \frac{1}{2} - \alpha }}{\left( {1 + 2\lambda } \right)^{ - \frac{1}{2} - 2\lambda }}\left[ {1 + \frac{1}{{1620{{\left( {1 + \lambda } \right)}^5}}}} \right] \times \notag\\
{\left[ {\frac{{\left( {1 + \lambda } \right) \csch\left( {\frac{1}{{1 + 2\lambda }}} \right)}}{{1 + 2\lambda }}} \right]^{\frac{1}{2} + \lambda }}{\left[ {\alpha \sinh \left( {\frac{1}{\alpha }} \right)} \right]^{ - \alpha /2}}{\left[ {\left( {1 + \lambda } \right)\sinh \left( {\frac{1}{{1 + \lambda }}} \right)} \right]^{\frac{{1 + \lambda }}{2}}}.
\end{gather}
\end{Theorem}
\begin{proof}
Observe first that $t^{\alpha} \le 1\,\forall t \in \bm{\Omega}_1$. Notice also that
\begin{equation}\label{eq:nmbz1}
\resizebox{0.5\textwidth}{!}{%
$\frac{1}{(n+1)! \Gamma(\alpha+1) \hK_{n+1}^{\lambda}} = \frac{{{2^{ - 2n - 1}} (2\lambda+1)\Gamma \left( {\lambda  + 2} \right)\Gamma \left( {n + 2\lambda  + 1} \right)}}{(\lambda+1) {\Gamma \left( {\alpha  + 1} \right)\Gamma \left( {2\lambda  + 2} \right)\Gamma \left( {n + 2} \right)\Gamma \left( {n + \lambda  + 1} \right)}},$
}
\end{equation}
by definition. The asymptotic inequality \eqref{eq:asymptineqsd1} results immediately from \eqref{eq:eeer1} after applying the sharp inequalities of the Gamma function \cite[Ineq. (96)]{elgindy2018optimal} on Eq. \eqref{eq:nmbz1} and using \cite[Lemma 5.1]{elgindy2018high}, which gives the uniform norm of Gegenbauer polynomials and their associated shifted forms.
\end{proof}

Theorem \ref{thm:kl1} manifests that the error bound is influenced by the smoothness of the function $f$ (through its derivatives) and the specific values of $\alpha$ and $\lambda$ with super-exponentially decay rate as $n \to \infty$, which guarantees that the error becomes negligible even for moderate values of $n$. By looking at the bounding constant $\vartheta_{\alpha,\lambda}$, we notice that, while holding $\lambda$ fixed, the factor \(\alpha^{-\frac{1}{2} - \alpha}\) decays exponentially as \(\alpha\) increases because \(\alpha^{-\alpha}\) dominates. For large \(\alpha\), \(\sinh(1/\alpha) \approx 1/\alpha\), so the factor $\left[\alpha \sinh(1/\alpha)\right]^{-\alpha/2}$ behaves like $1$. Thus, it does not significantly affect the behavior as \(\alpha\) increases. The factor \(e^{\alpha + 1}\) grows exponentially as \(\alpha\) increases. Combining these observations, the dominant behavior as \(\alpha\) increases is determined by the exponential growth of \(e^{\alpha + 1}\) and the exponential decay of \(\alpha^{-\frac{1}{2} - \alpha}\). The exponential decay dominates, so $\vartheta_{\alpha,\lambda}$ decays as \(\alpha\) increases, leading to a tighter error bound and improved convergence rate. On the other hand, considering large $\lambda$ values while holding $\alpha$ fixed, the factor \((1 + 2\lambda)^{-\frac{1}{2} - 2\lambda}\) decays as \(\lambda\) increases. The factor $1 + 1/\left({1620(1 + \lambda)^5}\right)$ approaches 1 as \(\lambda\) increases. For large \(\lambda\), \(\csch\left[1/(1 + 2\lambda)\right] \approx 1 + 2\lambda\), so the factor \(\left[\frac{(1 + \lambda) \csch\left(\frac{1}{1 + 2\lambda}\right)}{1 + 2\lambda}\right]^{\frac{1}{2} + \lambda}\) behaves like \((1 + \lambda)^{\frac{1}{2} + \lambda}\). This grows as \(\lambda\) increases. Finally, the factor \(\left[(1 + \lambda) \sinh\left(\frac{1}{1 + \lambda}\right)\right]^{\frac{1 + \lambda}{2}}\) behaves like \(1\) as $\lambda$ increases. Thus, it does not significantly affect the behavior as \(\lambda\) increases. Combining these observations, the dominant behavior as \(\lambda\) increases is determined by the growth of \((1 + \lambda)^{\frac{1}{2} + \lambda}\) and the decay of \((1 + 2\lambda)^{-\frac{1}{2} - 2\lambda}\). The decay of \((1 + 2\lambda)^{-\frac{1}{2} - 2\lambda}\) dominates, so \(\vartheta_{\alpha,\lambda}\) also decays as \(\lambda\) increases. It is noteworthy that \( \vartheta_{\alpha, \lambda} \) remains finite as \( \lambda \to -1/2 \), even though \( \lim_{\lambda \to -0.5^+} T_1(\lambda) = \infty \), where \( T_1(\lambda) = (1 + 2\lambda)^{-\frac{1}{2} - 2\lambda} \). This divergence is offset by the behavior of \( T_2(\lambda) \):
\begin{equation}
T_2(\lambda) = \left[ \frac{(1 + \lambda) \csch\left( \frac{1}{1 + 2\lambda} \right)}{1 + 2\lambda} \right]^{\frac{1}{2} + \lambda}.
\end{equation}
Specifically, for \( \lambda \to -0.5^+ \), using the approximation \( \csch(x) \approx 2e^{-x} \) for large \( x \):
\begin{equation}
T_2(\lambda) \approx \left[ \frac{(1 + \lambda) \cdot 2e^{-\frac{1}{1 + 2\lambda}}}{1 + 2\lambda} \right]^{\frac{1}{2} + \lambda} \sim \left[ \frac{0.5 \cdot 2e^{-\frac{1}{0^+}}}{0^+} \right]^{0^+} \to 0.
\end{equation}
Consequently,
\begin{equation}
\lim_{\lambda \to -0.5^+} T_1(\lambda) T_2(\lambda) = 0.
\end{equation}
Hence, \( \vartheta_{\alpha, \lambda} \to 0 \) as \( \lambda \to -1/2^+ \), driven by the product \( T_1(\lambda) T_2(\lambda) \to 0 \). For $\lambda \in \MBRmhzer$, we notice that $T_1$ is strictly concave with a maximum value at 
\begin{equation}\label{eq:lambdas1}
\lambda^* = \frac{-e + e^{\ProductLog(e/2)}}{2e} \approx -0.1351,
\end{equation}
rounded to 4 significant digits; cf. Theorem \ref{thm:dfjs1}. Figure \ref{fig:Fig5} shows further the plots of $T_2(\lambda)$ and 
\[T_3(\lambda) = \left[ (1+\lambda) \sinh\left(\frac{1}{1+\lambda}\right) \right]^{\frac{1+\lambda}{2}},\]
where $T_2$ grows at a slow, quasi-linearly rates of change, while $T_3$ decays at a comparable rate; thus, their product remains positive with a small bounded variation. This shows that, while holding $\alpha$ fixed, the bounding constant $\vartheta_{\alpha,\lambda}$ displays a unimodal behavior: it rises from 0 as $\lambda$ increases from $-1/2^+$, attains a maximum at $\lambda^* \approx -0.1351$, and subsequently decays monotonically for $\lambda > \lambda^*$. Figure \ref{fig:Fig3} highlights how $\vartheta_{\alpha, \lambda}$ decays with increasing $\alpha$ while exhibiting varying sensitivity to $\lambda$, with the $\lambda = \lambda^*$ case (red curve) representing the parameter value that maximizes $\vartheta_{\alpha, \lambda}$ for fixed $\alpha$. 

\begin{figure}[ht]
\centering
\includegraphics[scale=0.55]{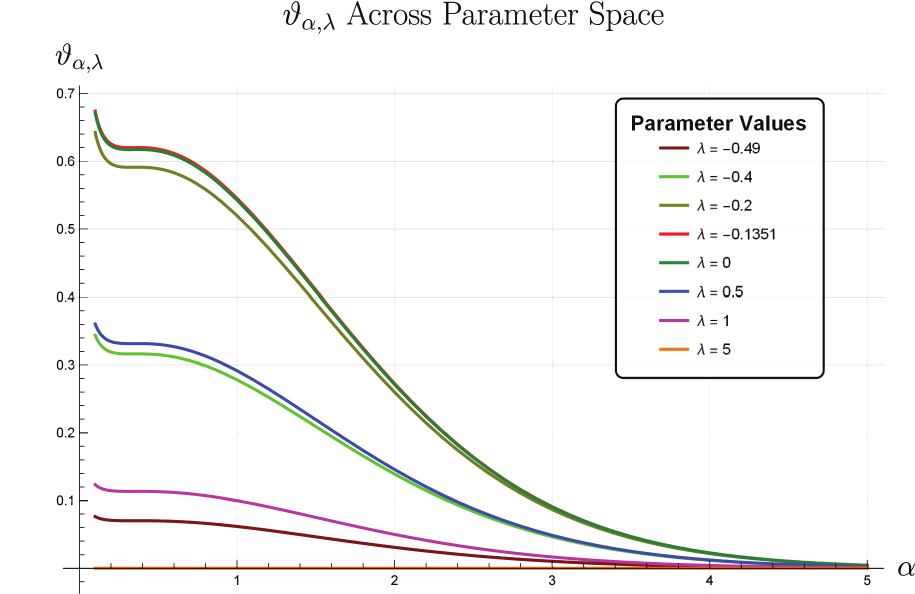}
\caption{Behavior of the bounding constant $\vartheta_{\alpha, \lambda}$ as $\alpha$ varies from $0.1$ to $5$, shown for eight representative values of $\lambda$ ($-0.49$, $-0.4$, $-0.2$, $-0.1351$, $0$, $0.5$, $1$, and $5$). Each curve corresponds to a distinct $\lambda$ value: \textcolor[rgb]{0.5,0,0}{dark red} ($\lambda = -0.49$), \textcolor[rgb]{0,1,0}{green} ($\lambda = -0.4$), \textcolor[rgb]{0.5,0.5,0}{olive} ($\lambda = -0.2$), \textcolor[rgb]{1,0,0}{bright red} ($\lambda = -0.1351$), \textcolor[rgb]{0,0.5,0}{dark green} ($\lambda = 0$), \textcolor[rgb]{0,0,1}{blue} ($\lambda = 0.5$), \textcolor[rgb]{1,0,1}{magenta} ($\lambda = 1$), and \textcolor[rgb]{1,0.5,0}{orange} ($\lambda = 5$).}
\label{fig:Fig3}
\end{figure}

\begin{figure}[ht]
\centering
\includegraphics[scale=0.5]{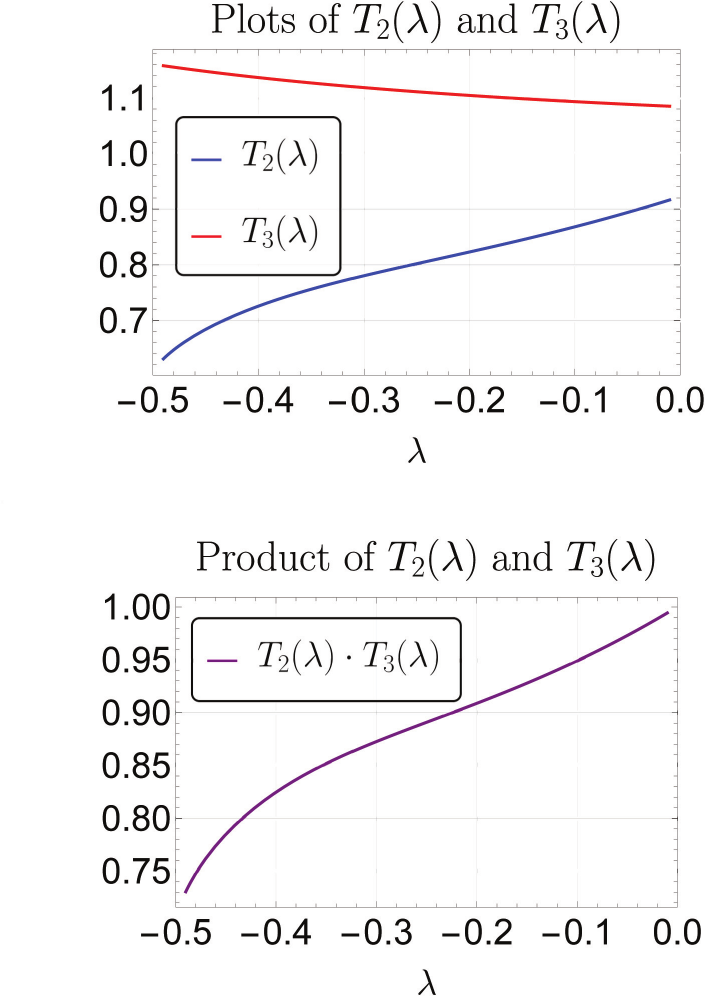}
\caption{(Top) Comparison of the functions $T_2$ (blue curve) and $T_3$ (red curve) over the interval $\MBRmhzer$. (Bottom) The product $T_2 T_3$ (purple), showing the combined behavior of the two functions. All plots demonstrate the dependence on the parameter $\lambda$ in the negative domain near zero.
}
\label{fig:Fig5}
\end{figure}

It is important to note that the bounding constant \(\vartheta_{\alpha,\lambda}\) modulates the error bound without altering its super-exponential decay as \(n \to \infty\). Near \(\lambda = -1/2\), \(\vartheta_{\alpha,\lambda} \to 0\), which shrinks the truncation error bound, improving accuracy despite potential sensitivity in the SG polynomials. At \(\lambda = \lambda^*\), the maximized \(\vartheta_{\alpha,\lambda}\) widens the bound, though it still vanishes as \(n \to \infty\). Beyond \(\lambda = \lambda^*\), the overall error bound decreases monotonically, despite a slower \(n\)-dependent decay for larger positive \(\lambda\). The super-exponential term \(\left(\tfrac{e}{4n}\right)^n\) ensures rapid convergence in all cases. The decreasing \(\vartheta_{\alpha,\lambda}\) after \(\lambda^*\), combined with the fact that the truncation error bound (excluding \(\vartheta_{\alpha,\lambda}\)) is smaller for \(-1/2 < \lambda \leq 0\) compared to its value for \(\lambda > 0\), further suggests that \(\lambda = 0\) appears ``optimal'' in practice for minimizing the truncation error bound \(\forall n\) in a stable numerical scheme, given the SG polynomial instability near \(\lambda = -1/2\). However, for relatively small or moderate values of \(n\), other choices of \(\lambda\) may be optimal, as the derived error bounds are asymptotic and apply only as \(n \to \infty\).

In what follows, we analyze the truncation error of the quadrature formula \eqref{eq:RL_quad}, and demonstrate how its results complement the preceding analysis.

\begin{Theorem}\label{subsec:err:thm3}
Let $j \in \MBJP_n, t \in \bm{\Omega}_1$, and assume that $\hG_j^{\lambda}\left(t (1 - y^{1/\alpha})\right)$ is interpolated by the SG polynomials with respect to the variable $y$ at the SGG nodes $\hat{t}_{n_q,0:n_q}^{\lambda_q}$. Then $\exists\,\eta = \eta(y) \in \IFOmega_1$ such that the truncation error, $\FR{T}_{j, n_q}^{\lambda_q}(\eta)$, in the quadrature approximation \eqref{eq:RL_quad} is given by
\begin{gather}
\FR{T}_{j, n_q}^{\lambda_q}(\eta) = \frac{(-1)^{n_q+1} \hchi_{j,n_q+1}^{\lambda}}{(n_q+1)! \hK_{n_q+1}^{\lambda_q}} \left(\frac{t}{\alpha}\right)^{n_q+1} \eta^{\frac{(n_q+1) (1-\alpha)}{\alpha}} \times \notag\\
 \hG_{j-n_q-1}^{\lambda+n_q+1}\left(t\left(1-\eta^{\frac{1}{\alpha}}\right)\right)\,\C{I}_1^{(y)} {\hG_{n_q+1}^{\lambda_q}} \cdot \MBI_{j \ge n_q+1},\label{subsec:err:eq:squadki1}
\end{gather}
where $\hchi_{n,m}^{\lambda}$ is defined by 
\begin{equation}
\hchi_{n,m}^{\lambda} = \frac{n! \Gamma(\lambda+1/2) \Gamma(n+m+2 \lambda)}{(n-m)! \Gamma(n+2 \lambda) \Gamma(m+\lambda+1/2)}\quad \forall n \ge m.\\\label{eq:hhkk1}
\end{equation}
\end{Theorem}
\begin{proof}
Let $\hG_j^{\lambda, m}$ denote the $m$th-derivative of $\hG_j^{\lambda}\,\forall j \in \MBJP_n$. \cite[Theorem 4.1]{Elgindy20161} tells us that
\begin{equation}\label{eq:Pr1oof1}
\FR{T}_{j, n_q}^{\lambda_q}(\eta) = \frac{1}{(n_q+1)! \hK_{n_q+1}^{\lambda_q}} \left[\NPart{y}{n_q+1} \hG_{j}^{\lambda}\left(t\left(1- y^{\frac{1}{\alpha}}\right)\right)\right]_{y=\eta} \C{I}_1^{(y)} {\hG_{n_q+1}^{\lambda_q}}
\end{equation}
The error bound \eqref{subsec:err:eq:squadki1} is accomplished by applying the Chain Rule on Eq. \eqref{eq:Pr1oof1}, which gives
\begin{gather}
\FR{T}_{j, n_q}^{\lambda_q}(\eta) = \frac{(-1)^{n_q+1}}{(n_q+1)! \hK_{n_q+1}^{\lambda_q}} \left(\frac{t}{\alpha}\right)^{n_q+1} \eta^{\frac{(n_q+1) (1-\alpha)}{\alpha}} \hG_j^{\lambda,n_q+1}\left(t\left(1-\eta^{\frac{1}{\alpha}}\right)\right)\,\C{I}_1^{(y)} {\hG_{n_q+1}^{\lambda_q}}.\label{eq:Pr1oof2}
\end{gather}
The proof is complete by realizing that 
\[\hG_j^{\lambda,n_q+1}\left(t\left(1-\eta^{\frac{1}{\alpha}}\right)\right) = \left.\NPart{\tau}{n_q+1} \hG_j^{\lambda}\left(\tau\right)\right|_{\tau = t\left(1-\eta^{\frac{1}{\alpha}}\right)} = 0,\]
$\forall j < n_q+1$.
\end{proof}

The next theorem provides an upper bound on the quadrature truncation error derived in Theorem \ref{subsec:err:thm3}. 

\begin{Theorem}\label{subsec:err:thm2}
Suppose that the assumptions of Theorem \ref{subsec:err:thm3} hold true. Then the truncation error, $\FR{T}_{j, n_q}^{\lambda_q}(\eta)$, in the quadrature approximation \eqref{eq:RL_quad} vanishes $\forall j < n_q+1$, and is bounded above by:
\begin{gather}\label{eq:SolySofy20251}
\left|\FR{T}_{j, n_q}^{\lambda_q}(\eta)\right| \simlteq \eta^{\frac{n_q (1-\alpha)}{\alpha}} \Upsilon_{\rho^{\lambda_q}}(n_q) \left\{ \begin{array}{l}
{\mu _{\lambda ,{\lambda _q}}} n_q^{{\lambda _q} - \lambda } \left(\frac{2 t}{\alpha}\right)^{n_q},\quad j \sim n_q, \\
\frac{\gamma_{n_q}^{\lambda} j^{2(n_q+1)}}{\Theta_{\lambda_q} n_q^{3/2 - \lambda_q}} \left(\frac{e t}{2 \alpha n_q}\right)^{n_q},\quad j \gg n_q,
\end{array} \right.
\end{gather}
$\forall j \ge n_q+1\,\forallL n_q$, where ${\mu _{\lambda ,{\lambda _q}}} = {{\nu ^\lambda }}/{{\Theta _{{\lambda _q}}}}$, $\Theta _{{\lambda _q}}$ is as defined by \eqref{eq:qwqwq2}, $\nu ^\lambda \in \MBRP$ is a $\lambda$-dependent constant, $\rho^{\lambda_q} > 1$ is a $\lambda_q$-dependent constant, and $\gamma_{n_q}^{\lambda}$ is a constant dependent on $n_q$ and $\lambda$. 
\end{Theorem}
\begin{proof}
Notice first that
\[\left|\hG_{j-n_q-1}^{\lambda+n_q+1}\left(t\left(1-\eta^{\frac{1}{\alpha}}\right)\right)\right| \le 1,\]
by \cite[Lemma 5.1]{elgindy2018high}, since $\lambda+n_q+1 > 0$. We now consider the following two cases.
\begin{description}
\item[Case I ($j \sim n_q$):] Lemmas \ref{lem:LaElahaElaA1} and \ref{lem:2} imply:
\begin{equation}
\frac{\hchi_{j, n_q + 1}^\lambda}{(n_q + 1)! \, \hK_{n_q + 1}^{\lambda_q}} 
\simlt \frac{\nu^{\lambda} \left(\frac{4}{e}\right)^{n_q} n_q^{n_q + \frac{3}{2} - \lambda}}{\Theta_{\lambda_q} n_q^{3/2 - \lambda_q} 
    \left( \frac{2n_q}{e} \right)^{n_q}} = {\mu _{\lambda ,{\lambda _q}}}{2^{{n_q}}}n_q^{{\lambda _q} - \lambda }.\label{eq:xcxcxc1}
\end{equation}
\item[Case II ($j \gg n_q$):] Here $\hchi_{j, n_q + 1}^\lambda = O\left(j^{2(n_q+1)}\right)$ by Lemma \ref{lem:2}, and we have
\begin{equation}\label{eq:wewwew1}
\frac{\hchi_{j, n_q + 1}^\lambda}{(n_q + 1)! \, \hK_{n_q + 1}^{\lambda_q}} \simlt \frac{\gamma_{n_q}^{\lambda} j^{2(n_q+1)}}{\Theta_{\lambda_q} n_q^{3/2 - \lambda_q} 
    \left( \frac{2n_q}{e} \right)^{n_q}}.
\end{equation}
\end{description}
Formula \eqref{eq:SolySofy20251} is obtained by substituting \eqref{eq:xcxcxc1} and \eqref{eq:wewwew1} into \eqref{subsec:err:eq:squadki1}.
\end{proof}
When $j \sim n_q$, the dominant term in $\sup \left|\FR{T}_{j, n_q}^{\lambda_q}(\eta)\right|$ becomes:
\[
{2^{{n_q}}}n_q^{{\lambda _q} - \lambda } \cdot \left(\frac{t}{\alpha}\right)^{n_q} \eta^{\frac{n_q\,(1-\alpha)}{\alpha}} = {\left( {\frac{{2\,t\,{\eta ^{\frac{1}{\alpha } - 1}}}}{\alpha }} \right)^{{n_q}}}n_q^{{\lambda _q} - \lambda }.
\]
Exponential decay occurs when:
\begin{equation}\label{eq:vbvbvbv1}
\alpha > 2\,t\,{\eta ^{\frac{1}{\alpha } - 1}}.
\end{equation}
On the other hand, the dominant term in the error bound $\forall j \gg n_q$ is given by:
\begin{gather}
\frac{\gamma_{n_q}^{\lambda} j^{2(n_q+1)}}{\Theta_{\lambda_q} n_q^{3/2 - \lambda_q} 
    \left( \frac{2n_q}{e} \right)^{n_q}} \cdot \left(\frac{t}{\alpha}\right)^{n_q} \eta^{\frac{n_q\,(1-\alpha)}{\alpha}}\notag\\
= \frac{{\gamma _{{n_q}}^\lambda }}{{{\Theta _{{\lambda _q}}}}}{\left( {\frac{{e t {j^2} {\eta ^{\frac{1}{\alpha}-1}}}}{{2\alpha {n_q}}}} \right)^{{n_q}}}{j^2}n_q^{{\lambda _q} - 3/2}.
\end{gather}
For convergence, we require:
\[
{\frac{{e t {j^2} {\eta ^{\frac{1}{\alpha}-1}}}}{{2\alpha {n_q}}}} < 1.
\]
Given $j \gg n_q, j^2 / n_q \to \infty$, and ${\frac{{e t {j^2} {\eta ^{\frac{1}{\alpha}-1}}}}{{2\alpha {n_q}}}}$ typically diverges unless 
\begin{equation}\label{eq:vdhvsdj1}
{\frac{{t\,{\eta ^{\frac{1}{\alpha}-1}}}}{{\alpha}}}\,\text{is sufficiently small}:\quad {\left( {\frac{{e t {j^2} {\eta ^{\frac{1}{\alpha}-1}}}}{{2\alpha {n_q}}}} \right)^{{n_q}}}{j^2}n_q^{{\lambda _q} - 3/2} \to 0.
\end{equation}
The relative choice of $\lambda$ and $\lambda_q$ in either case controls the error bound's decay rate. In particular, choosing $\lambda_q < \lambda$ ensures faster convergence rates when $j \sim n_q$ due to presence of the polynomial factor $n_q^{\lambda_q-\lambda}$. For $j \gg n_q$, choosing $\lambda_q < 3/2$ accelerates the convergence if Condition \eqref{eq:vdhvsdj1} holds.

The following theorem provides a rigorous asymptotic bound on the total truncation error for the RLFI approximation, combining both sources of error, namely, the interpolation and quadrature errors. 

\begin{Theorem}[Asymptotic Total Truncation Error Bound]\label{thm:TotErrKKK1}
Suppose that $f \in C^{n+1}(\bm{\Omega}_1)$ is approximated by the GBFA interpolant \eqref{eq:RL_Lagint}, and the assumptions of Theorem \ref{subsec:err:thm3} hold true. Then the total truncation error in the RLFI approximation of $f$, denoted by ${}^{\alpha}\ME{E}_{n, n_q}^{\lambda, \lambda_q}(t, \xi, \eta)$, arising from both the series truncation \eqref{eq:RL_Lagint} and the quadrature approximation \eqref{eq:RL_quad}, is asymptotically bounded above by:
\begin{gather}
\left|{}^{\alpha}\ME{E}_{n, n_q}^{\lambda, \lambda_q}(t, \xi, \eta)\right| \simlt \C{A}_{n+1}\,\vartheta_{\alpha, \lambda} {\left( {\frac{e}{4}} \right)^n}{n^{ - \frac{3}{2} - n + \lambda }}\,\Upsilon_{\sigma^{\lambda}}(n) \notag\\
+ \frac{\C{A}_0\,\varpi^{\text{upp}} t^{\alpha}}{\lambdabar_{\max}^{\lambda}\,\Gamma(\alpha+1)} \,n (n-n_q) \eta^{\frac{n_q (1-\alpha)}{\alpha}} 
{}_2\Upsilon_{D^{\lambda},\rho^{\lambda_q}}(n,n_q)\,\MFI_{n \ge n_q+1} \times\\
\left\{ \begin{array}{l}
{\mu _{\lambda ,{\lambda _q}}} n_q^{{\lambda _q} - \lambda } \left(\frac{2 t}{\alpha}\right)^{n_q},\quad n \sim n_q,\\
\frac{\gamma_{n_q}^{\lambda} n^{2(n_q+1)}}{\Theta_{\lambda_q} n_q^{3/2 - \lambda_q}} \left(\frac{e t}{2 \alpha n_q}\right)^{n_q},\quad n \gg n_q,
\end{array} \right.\label{eq:erere1}
\end{gather}
$\forallL n, n_q$, where 
\begin{equation}
\varpi^{\text{upp}} = \left\{ \begin{array}{l}
\varpi^{\text{upp},+},\quad \forall \lambda \in \MBRzerP,\\
\varpi^{\text{upp,-}},\quad \lambda \in \MBRmhzer,\\
\end{array} \right.
\end{equation}
\begin{equation}
\frac{1}{\lambdabar_{\max}^{\lambda}} = \left\{ \begin{array}{l}
\displaystyle{\frac{1}{\lambdabar_n^{\lambda}}},\quad \lambda \in \MBRzerP,\\
\displaystyle{\frac{1}{\lambdabar_{n_q+1}^{\lambda}}},\quad \lambda \in \MBRmhzer,
\end{array} \right.
\end{equation}
\begin{equation}
{}_2\Upsilon_{D^{\lambda},\rho^{\lambda_q}}(n,n_q) = \begin{cases}
1, & \lambda \in \MBRzerP,\quad \lambda_q \in \MBRzerP, \\
D^{\lambda} n^{-\lambda}, & \lambda \in \MBRmhzer,\quad \lambda_q \in \MBRzerP, \\
\rho^{\lambda_q} n_q^{-\lambda_q}, & \lambda \in \MBRzerP,\quad \lambda_q \in \MBRmhzer, \\
D^{\lambda} \rho^{\lambda_q} n^{-\lambda} n_q^{-\lambda_q}, & \lambda \in \MBRmhzer,\quad \lambda_q \in \MBRmhzer,
\end{cases}
\end{equation}
with $D^{\lambda} > 1$ being a $\lambda$-dependent constant, and $\vartheta_{\alpha,\lambda}, \Upsilon_{\sigma^{\lambda}}(n), \gamma_{n_q}^{\lambda}$, and $\Theta_{\lambda_q}$ are constants with the definitions and properties outlined in Theorems \ref{thm:kl1} and \ref{subsec:err:thm2}, and Lemmas \ref{lem:LaElahaElaA1} and \ref{lem:2}, and $\varpi^{\text{upp},\pm}$ are as defined by \cite[Formulas (B.4) and (B.15)]{Elgindy2023d}.
\end{Theorem}
\begin{proof}
The total truncation error combines the interpolation error from Theorem \ref{thm:kl1} and the accumulated quadrature errors from Theorem \ref{subsec:err:thm2} for $j \in \MBJP_n$:
\begin{gather*}
{}^{\alpha}\ME{E}_{n, n_q}^{\lambda, \lambda_q}(t, \xi, \eta) = {}^{\alpha}\ME{T}_n^{\lambda}(t, \xi) \\
+ \frac{t^{\alpha}}{\Gamma(\alpha+1)} \sum_{k \in \MBJP_n} {\hvarpi_k^{\lambda} f_k \sum_{j \in \MBJP_n} {\left(\hlambdabar_j^{\lambda}\right)^{-1} \FR{T}_{j, n_q}^{\lambda_q}(\eta)\,\hG_j^{\lambda} \left(\hx_{n,k}^{\lambda}\right)}}\\
= {}^{\alpha}\ME{T}_n^{\lambda}(t, \xi) + \frac{t^{\alpha}}{\Gamma(\alpha+1)} \sum_{k \in \MBJP_n} {\varpi_k^{\lambda} f_k \sum_{j \in \MBN_{n_q+1,n}} {\left(\lambdabar_j^{\lambda}\right)^{-1} \FR{T}_{j, n_q}^{\lambda_q}(\eta)\,\hG_j^{\lambda} \left(\hx_{n,k}^{\lambda}\right)}}.
\end{gather*}
Using the bounds on Christoffel numbers $\varpi_k^{\lambda}$ and normalization factors $\lambdabar_j^{\lambda}$ from \cite[Lemmas B.1 and B.2]{Elgindy2023d}, along with the uniform bound on Gegenbauer polynomials from \cite[Lemma 5.1]{elgindy2018high}, we obtain:
\begin{gather*}
\left|{}^{\alpha}\ME{E}_{n, n_q}^{\lambda, \lambda_q}(t, \xi, \eta)\right| \simlt \left|{}^{\alpha}\ME{T}_n^{\lambda}(t, \xi)\right| \\
+ \frac{\C{A}_0\,\varpi^{\text{upp}} t^{\alpha}}{\lambdabar_{\max}^{\lambda}\,\Gamma(\alpha+1)} (n+1) (n-n_q) \max_{j \in \MBN_{n_q+1:n}} \left|\FR{T}_{j, n_q}^{\lambda_q}(\eta)\right| \Upsilon_{D^{\lambda}}(n).\label{eq:ppoo1}
\end{gather*}
The proof is completed by applying Theorems \ref{thm:kl1} and \ref{subsec:err:thm2} on Formula \eqref{eq:ppoo1}, noting that $\max_{j \in \MBN_{n_q+1:n}} \left|\FR{T}_{j, n_q}^{\lambda_q}(\eta)\right|$ occurs at $j = n$.
\end{proof}
The total truncation error bound presented in Theorem \ref{thm:TotErrKKK1} reveals several important insights about the convergence behavior of the RLFI approximation: (i) The total error consists of the interpolation error term ${}^{\alpha}\ME{T}_n^{\lambda}(t, \xi)$ and the accumulated quadrature error term. The interpolation error term decays at a super-exponential rate $\forallL n$ due to the factor $\left(\tfrac{e}{4n}\right)^n$. The quadrature error either vanishes when $n_q \ge n$, decays exponentially when $n_q \sim n$ under Condition \eqref{eq:vbvbvbv1}, or typically diverges when $n \gg n_q$ unless Condition \eqref{eq:vdhvsdj1} is fulfilled. Therefore, the quadrature nodes should scale appropriately with the interpolation mesh size in practice. (ii) The interpolation error bound tightens as $\alpha$ increases, due to the $\vartheta_{\alpha,\lambda}$ factor's decay. The $\lambda$ parameter shows a unimodal influence on the interpolation error, with potentially maximum error size at about $\lambda^* \approx -0.1351$. The $\lambda_q$ parameter should generally be chosen smaller than $\lambda$ to accelerate the convergence of the quadrature error when $n_q \sim n$. This analysis suggests that the proposed method achieves spectral convergence when the parameters are chosen appropriately. The quadrature precision should be selected to maintain balance between the two error components based on the desired accuracy and computational constraints.

\subsection{Practical Guidelines for Parameter Selection}
\label{subsec:PGFPS1}
The asymptotic analysis reveals dependencies on the parameters $\lambda$ (for interpolation) and $\lambda_q$ (for quadrature), which play crucial roles in determining the method's accuracy and efficiency. Here, we provide practical guidance for selecting these parameters to balance interpolation error, quadrature error, and computational cost.

The effectiveness of the GBFA method for approximating the RLFI hinges on the appropriate selection of parameters $\lambda$ and $\lambda_q$. Building upon established numerical approximation principles, our analysis incorporates specific considerations for RLFI computation. In particular, we identify the following effective operational range for the SG parameters $\lambda$ and $\lambda_q$:
\begin{equation}\label{eq:GPCIOC}
    \C{T}_{c,r} = \left\{ \gamma \mid -\frac{1}{2} + \varepsilon \leq \gamma \leq r, \, 0 < \varepsilon \ll 1, \, r \in \bm{\Omega}_{1,2} \right\}.
\end{equation}
This range, previously recommended by \citet{elgindy2020distributed} based on extensive theoretical and numerical testing consistent with broader spectral approximation theory, helps avoid numerical instability caused by increased extrapolation effects associated with larger positive SG indices. Furthermore, SG polynomials exhibit blow-up behavior as their indices approach $-0.5^+$. Within $\C{T}_{c,r}$, we observe a crucial balance between theoretical convergence and numerical stability for RLFI computation, a finding corroborated by our numerical investigations using the GBFA method, which consistently demonstrate superior performance across various test functions with parameter choices in this range.

Based on this analysis of error bounds and numerical performance, we recommend the following parameter selection strategies for RLFI approximation. 

\begin{itemize}
\item $\forallS n$ and $n_q$, the selection of $\lambda \in \C{T}_{c,r}$ is feasible. Furthermore, choosing smaller $\lambda_q \in \C{T}_{c,r}$ generally improves quadrature accuracy, with the notable exception of $\lambda_q = 0.5$, where the quadrature error is often minimized, as we demonstrate later in Figures \ref{fig:Fig1} and \ref{fig:Fig2}.
    \item $\forallL n$ and $n_q$:
\begin{itemize}
\item \textbf{For precision computations:} Select \(\lambda \in \bm{\Omega}_{-0.5+\varepsilon, 0} \setminus N_\delta(\lambda^*)\) and \(\lambda_q \in \C{T}_{c,r}\), where:
\begin{equation}
\lambda_q < \begin{cases} 
\lambda, & \text{if } n \sim n_q, \\
3/2, & \text{if } n \gg n_q,
\end{cases}
\end{equation}
with \(\lambda^*\) defined by Eq.~\eqref{eq:lambdas1}. Here, \(\bm{\Omega}_{-0.5+\varepsilon, 0}\) is a subset of the recommended interval $T_{c,r}$ for SG indices. Positive values of \(\lambda\) are excluded from \(\C{T}_{c,r}\), as the polynomial error factor \(n^{-\frac{3}{2} - n + \lambda}\) increases with positive \(\lambda\), as shown in Ineq.~\eqref{eq:asymptineqsd1}. The \(\delta\)-neighborhood \(N_\delta(\lambda^*)\) is excluded because \(\vartheta_{\alpha,\lambda}\), the leading factor in the asymptotic interpolation error bound, peaks at \(\lambda = \lambda^*\), potentially increasing the error. By avoiding this neighborhood, parameter choices that could amplify interpolation errors are circumvented, ensuring robust performance for high-precision applications.
    \item \textbf{For standard computational scenarios:} Utilize $\lambda = \lambda_q = 0$, which corresponds to shifted Chebyshev approximation. This recommendation employs the well-established optimality of Chebyshev approximation for smooth functions, offering a robust default that balances accuracy and efficiency for RLFI approximation.
\end{itemize}
\end{itemize}
This parameter selection guideline provides practical and effective guidance for balancing interpolation and quadrature errors and computational cost across a wide range of applications.

\section{Further Numerical Simulations}
\label{sec:FNS}
\textbf{Example 1.} To demonstrate the accuracy of the derived numerical approximation formulas, we consider the power function $f(t) = t^N$, where $N \in \mathbb{Z}^+$, as our first test case. The RLFI of $f$ is given analytically by
\begin{align*}
    \RLI{t}{\alpha}{f} = \frac{N!}{\Gamma(N + \alpha + 1)} t^{N+\alpha}.
\end{align*}
Figure~\ref{fig:Fig1} displays the logarithmic absolute errors of the RLFI approximations computed using the GBFA method, with fractional order $\alpha = 0.5$ evaluated at $t = 0.5$. The figure consists of four subplots that investigate: (i) The effects of varying the parameters $\lambda$, $\lambda_q$, and $n_q$, and (ii) a comparative analysis between the GBFA method and MATLAB's \texttt{integral} function with tolerance parameters set to $\texttt{RelTol} = \texttt{AbsTol} = 10^{-15}$. Our numerical experiments reveal several key observations:
\begin{itemize}
    \item Variation of $\lambda$ while holding other parameters constant ($\forallS n, n_q$) shows negligible impact on the error. The error reaches near machine epsilon precision at $n = 3$, consistent with Theorems~\ref{eq:gsadffhj1} and~\ref{subsec:err:thm3}, which predict the collapse of both interpolation and quadrature errors when $f^{(n+1)} \equiv 0$ and $n_q > n$.
    \item For $n \geq 5$, the total error reduces to pure quadrature error since $f^{(n+1)} \equiv 0$ while $n_q < n$.
    \item Variation of $\lambda_q$ significantly affects the error, with $\lambda_q \leq \lambda$ generally yielding higher accuracy.
    \item Increasing either $n$ or $n_q$ while fixing the other parameter leads to exponential error reduction.
\end{itemize}

The GBFA method achieves near machine-precision accuracy with parameter values $\lambda = \lambda_q = 0.5$ and $n_q = 12$, outperforming MATLAB's \texttt{integral} function by nearly two orders of magnitude. The method demonstrates remarkable stability, as evidenced by consistent error trends for $\lambda_q \leq \lambda$, with nearly exact approximations obtained for $n_q \geq 12$ in optimal parameter ranges. 

Figure \ref{fig:Fig8} compares further the computation times of the GBFA method and MATLAB's \texttt{integral} function, plotted on a logarithmic scale. The GBFA method demonstrates significantly lower computational times compared to MATLAB's \texttt{integral} function. This highlights the efficiency of the GBFA method, which achieves high accuracy with minimal computational cost.

\begin{figure*}[ht]
    \centering
    \includegraphics[width=0.5\textwidth]{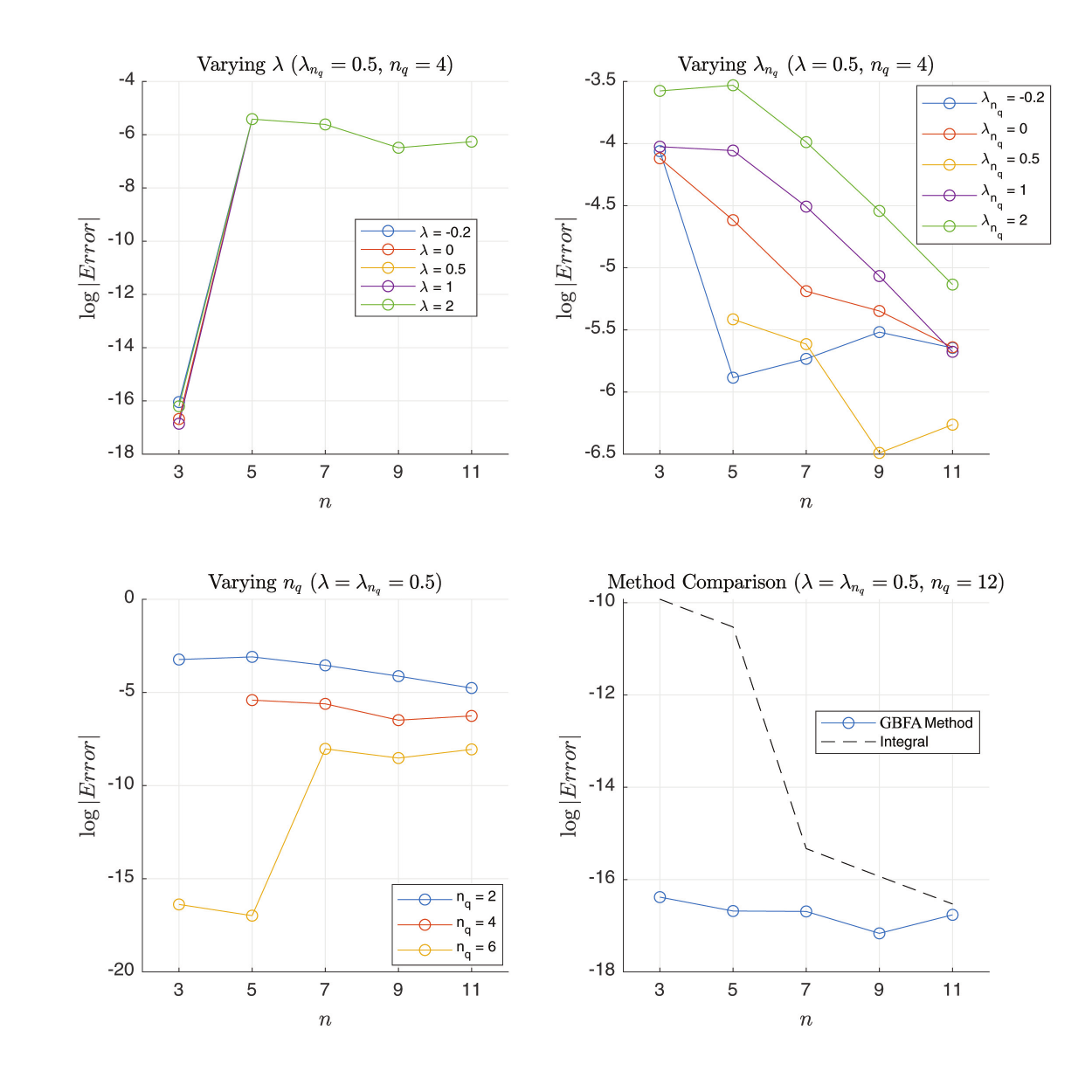}
    \caption{Logarithmic absolute errors of the RLFI approximations for the power function $f$, computed using the GBFA method. The fractional order is set to $\alpha = 0.5$, and approximations are evaluated at $t = 0.5$. Gegenbauer interpolant degrees match the function's degrees for $n = 3:2:11$. The figure presents errors under different conditions: \textit{Top-left}: Varying $\lambda$ with fixed $\lambda_q = 0.5$ and $n_q = 4$. \textit{Top-right}: Varying $\lambda_q$ with fixed $\lambda = 0.5$ and $n_q = 4$. \textit{Bottom-left}: Varying $n_q$ with $\lambda = \lambda_q = 0.5$. \textit{Bottom-right}: Comparison between RLIM and MATLAB's \texttt{integral} function using $n_q = 12$ and $\lambda = \lambda_q = 0.5$. ``Error'' refers to the difference between the true RLFI value and its approximation. Missing colored lines indicate zero error (approximation exact within numerical precision).}
    \label{fig:Fig1}
\end{figure*}

\begin{figure}[ht]
    \centering
    \includegraphics[width=0.45\textwidth]{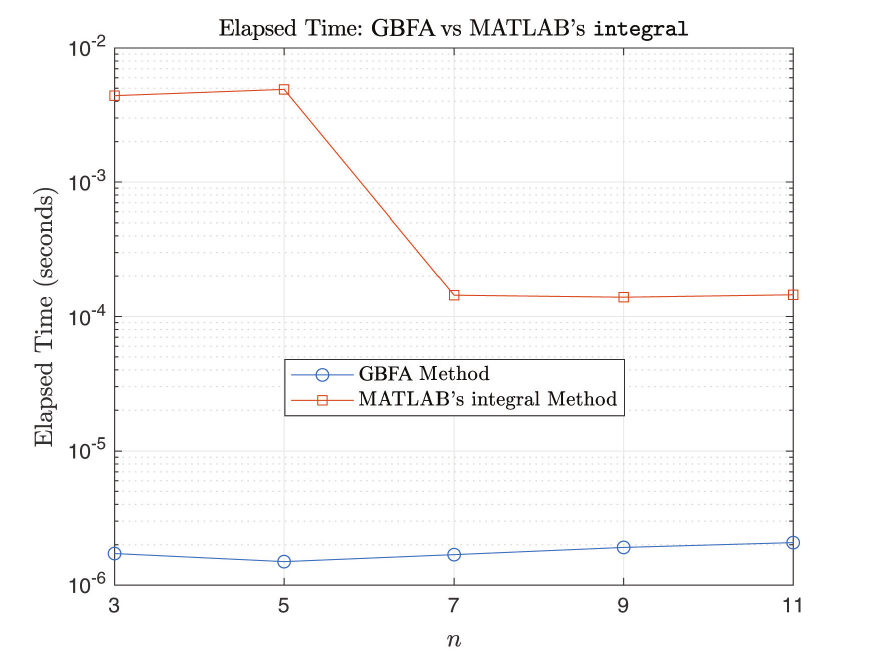}
    \caption{Comparison of the elapsed computation times (displayed on a logarithmic scale) for the GBFA Method and MATLAB's \texttt{integral} Method as a function of the polynomial degree $n$. All experiments were performed with $t = 0.5, \alpha = 0.5, \lambda_1 = \lambda_2 = 0.5$, and $n_q = 12$.}
    \label{fig:Fig8}
\end{figure}

\textbf{Example 2.} Next, we evaluate the $0.5$th-order RLFI of $g(t) = e^{k t}$, where $k \in \mathbb{R}_{\neq 0}$, over $\Omega_{0.5}$. The analytical solution on $\Omega_t$ is  
\[
\RLI{t}{\alpha}{g} = \frac{t^{\alpha} e^{k t}}{\Gamma(\alpha+1)} \, {}_1F_1(\alpha; \alpha+1; -k t).
\]  
Figure~\ref{fig:Fig2} presents the logarithmic absolute errors of the GBFA method, with subplots analyzing: (i) the impact of $\lambda$, $\lambda_q$, $n$, and $n_q$, and (ii) a performance comparison against MATLAB's \texttt{integral} function (tolerances $10^{-15}$). Some of the main findings include:  
\begin{itemize}  
    \item Similar to \textbf{Example 1}, varying $\lambda$ (with fixed small $n$, $n_q$) has minimal effect on accuracy, whereas $\lambda_q \leq \lambda$ consistently improves precision.  
    \item Exponential error decay occurs when increasing either $n$ or $n_q$ while holding the other constant.  
    \item Near machine-epsilon accuracy is achieved for $\lambda = \lambda_q = 0.5$ and $n_q = 12$, with the GBFA method surpassing \texttt{integral} by two orders of magnitude.  
\end{itemize}  
The method’s stability is further demonstrated by the uniform error trends for $\lambda_q \leq \lambda$ and its rapid exponential convergence (see Figure~\ref{fig:Fig6}), underscoring its suitability for high-precision fractional calculus, particularly in absence of closed-form solutions. 

Notably, this test problem was previously studied in \cite{dimitrov2021approximations} and \cite{ciesielski2024numerical} on $\bm{\Omega}_2$. The former employed an asymptotic expansion for trapezoidal RLFI approximation, while the latter used linear, quadratic, and three cubic spline variants, with all computations performed in 128-bit precision. The methods were tested for grid sizes \( N \in \{40, 80, 160, 320, 640\} \) (step sizes \( \Delta x \in \{0.05, 0.025, \dots, 0.003125\} \)). At \( N = 640 \), Dimitrov's method achieved the smallest error of \( 2.34 \times 10^{-13} \), whereas \citet{ciesielski2024numerical} reported their smallest error of \( 9.17 \times 10^{-13} \) using Cubic Spline Variant 1. With $\lambda = \lambda_q = 0.5$ and $n_q = 12$, our method attains errors close to machine epsilon ($\sim 10^{-16}$), surpassing both methods in \cite{dimitrov2021approximations} and \cite{ciesielski2024numerical} by several orders of magnitude, even with significantly fewer computational nodes.  

\begin{figure}[ht]
    \centering
    \includegraphics[width=0.5\textwidth]{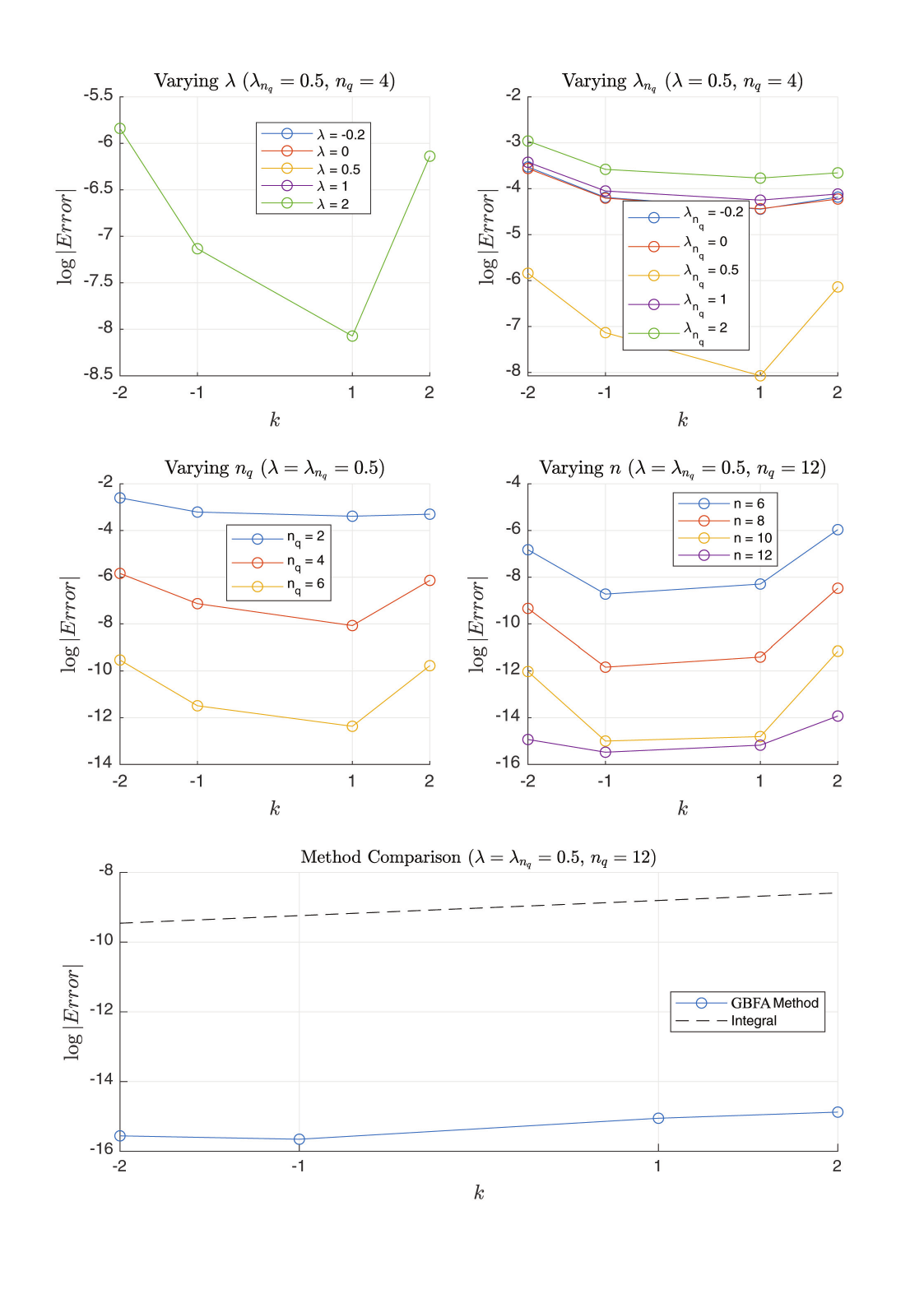}
    \caption{Logarithmic absolute errors of the RLFI approximations for the exponential function $g$, computed using the GBFA method. The fractional order is set to $\alpha = 0.5$, and the approximations are evaluated at $t = 0.5$ using a 13th-degree Gegenbauer interpolant for $k = -2, -1, 1, 2$, except for the middle right subplot, where $n$ varies. The figure presents errors under different conditions: \textit{Top-left}: Varying $\lambda$ with fixed $\lambda_q = 0.5$ and $n_q = 4$. \textit{Top-right}: Varying $\lambda_q$ with fixed $\lambda = 0.5$ and $n_q = 4$. \textit{Middle-left}: Varying $n_q$ with $\lambda = \lambda_q = 0.5$. \textit{Middle-right}: Varying $n$ with $\lambda = \lambda_q = 0.5$ and $n_q = 12$. \textit{Bottom}: Comparison between RLIM and MATLAB's \texttt{integral} function using $n_q = 12$ and $\lambda_1 = \lambda_2 = 0.5$. ``Error'' refers to the difference between the true RLFI value and its approximation.}
    \label{fig:Fig2}
\end{figure}

\begin{figure}[ht]
\centering
\includegraphics[scale=0.5]{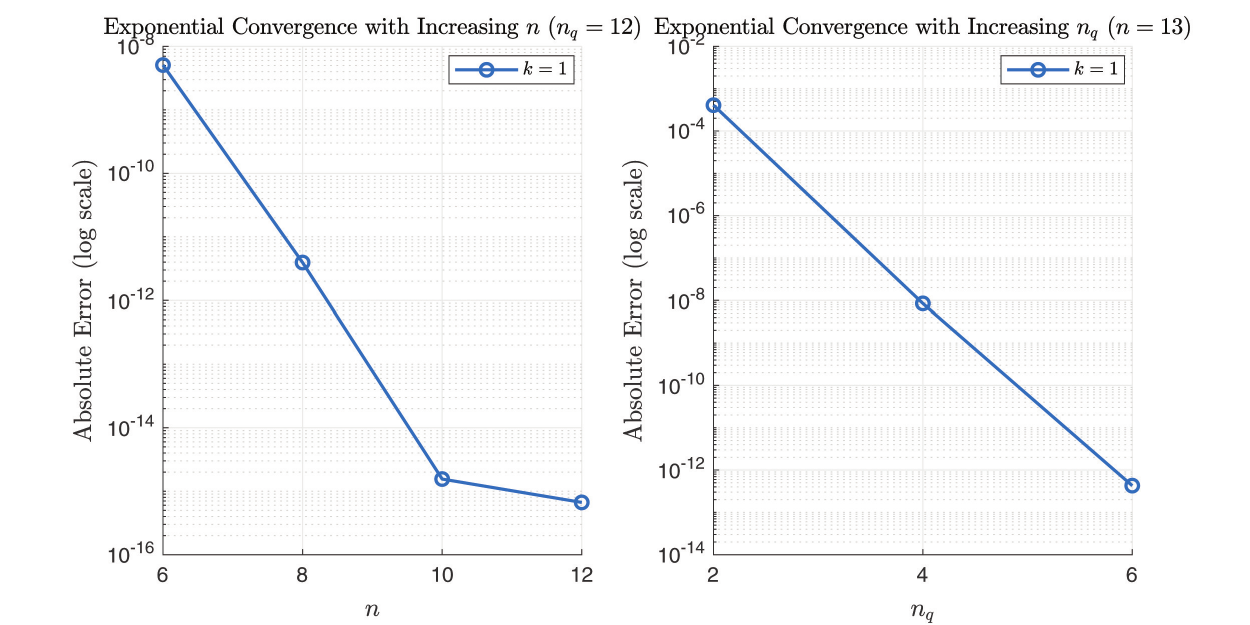}
\caption{The left subplot demonstrates the exponential convergence of the absolute error as $n$ increases, with $n_q$ fixed at 12. The right subplot illustrates the exponential convergence of the absolute error as the $n_q$ parameter increases, with $n$ fixed at 13. Both subplots show the absolute error on a logarithmic scale.}
\label{fig:Fig6}
\end{figure}

\begin{figure}[ht]
\centering
\includegraphics[scale=0.65]{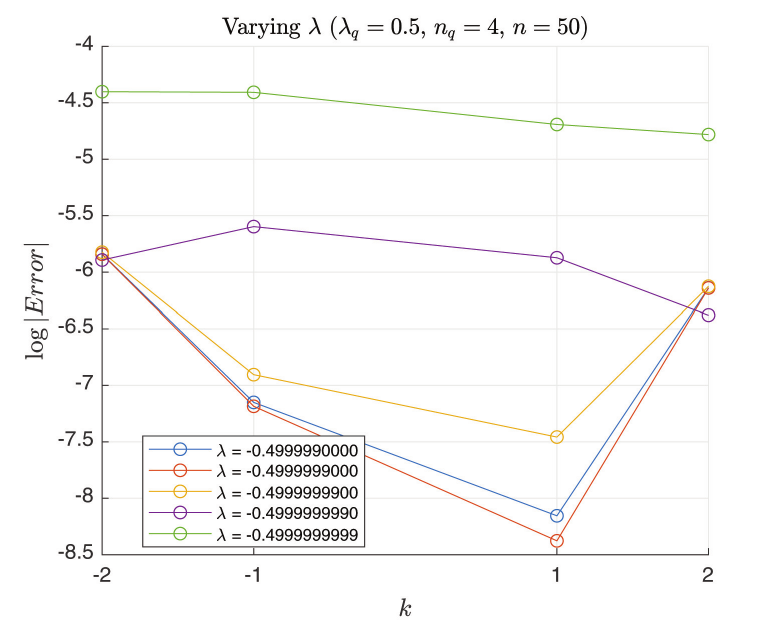}
\caption{The logarithmic absolute errors in the numerical evaluation of the RLFI of the function $g$ at $t = 0.5$ for various values of $\lambda$ extremely close to $-0.5$, using $n = 50, n_q = 4$, and $\lambda_q = 0.5$. Each curve corresponds to a different $\lambda$ parameter, demonstrating the sensitivity of the GBFA method as $\lambda$ approaches $-0.5$.}
\label{fig:Fig7}
\end{figure}

Section~\ref{subsec:PGFPS1} notes the potential instability of SG polynomials as $\lambda \to -1/2$, a behavior not yet investigated within the GBFA framework. Figure~\ref{fig:Fig7} visually confirms this trend: initially, the errors decrease as $\lambda$ approaches $-0.5$, since the leading error factor $\vartheta_{\alpha, \lambda} \to 0$ as $\lambda \to -1/2^+$, as explained by Theorem~\ref{thm:kl1} and evident when $\lambda$ transitions from $-0.499999$ to $-0.4999999$. However, the ill-conditioning of SG polynomials becomes prominent as $\lambda$ approaches $-0.5$ more closely, leading to larger errors. This non-monotonic behavior near $\lambda = -0.5$ underscores the practical stability limits of the method, aligning with the early observations in \cite{Elgindy2013b} on the sensitivity of Gegenbauer polynomials with negative parameters. In particular, \citet{Elgindy2013b} discovered that for large degrees $M$, Gegenbauer polynomials $G_M^{\lambda}(x)$ with $\lambda < 0$ exhibit pronounced sensitivity to small perturbations in $x$, unlike Legendre or Chebyshev polynomials. For example, \cite{Elgindy2013b} demonstrated that evaluating $G_{150}^{-1/4}(x)$ at $x = 1/2$ (exact: $5.754509478448837$) versus $x = 0.5001$ (perturbed by $10^{-4}$) introduces an absolute error of $\approx 0.0122$, whereas Legendre and Chebyshev polynomials of the same degree incur errors of only $\approx 10^{-4}$. This highlights the risks of using SG polynomials with $\lambda \approx -1/2$ for high-order approximations, where even minor numerical perturbations can disrupt spectral convergence. These results reinforce the need to avoid such parameter regimes for robust approximations, as prescribed in Section~\ref{subsec:PGFPS1}. 

\textbf{Example 3.} We evaluate the $0.5$th-order RLFI of the function \( f(t) = 2t^3 + 8t \) over the interval \( \Omega_{0.5} \). The analytical solution on \( \Omega_t \) is given by
\[
\RLI{t}{0.5}{f} = \frac{192 t^{7/2} + 1120 t^{3/2}}{105 \sqrt{\pi}},
\]
which serves as the reference for numerical comparison. Table~\ref{tab:Tab1new} presents a comprehensive comparison of numerical results obtained using three distinct approaches: (i) The proposed GBFA method, (ii) MATLAB's high-precision \texttt{integral} function, and (iii) MATHEMATICA's \texttt{NIntegrate} function. This test case has been previously investigated by \cite{batiha2023n} on the interval \( \Omega_2 \). Their approach employed an \( n \)-point composite fractional formula derived from a three-point central fractional difference scheme, incorporating generalized Taylor's theorem and fundamental properties of fractional calculus. Notably, their implementation with \( n = 10 \) subintervals achieved a relative error of approximately \( -6.87 \times 10^{-3} \), which is significantly higher than the errors produced by the current method.

\begin{table}[h!]
\caption{Comparison of the exact $0.5$th-order RLFI of $f(t) = 2 t^3+8 t$ over the interval $[0,0.5]$, the GBFA method, MATLAB's \texttt{integral} function, and MATHEMATICA \texttt{NIntegrate} function approximations. The table also includes the CPU times for computing the GBFA method and and MATLAB's \texttt{integral} function. All relative error approximations are rounded to 16 significant digits. The CPU times were computed in seconds (s).}
\centering
\resizebox{0.5\textwidth}{!}{%
\begin{tabular}{lcccc}
\toprule
\textbf{Quantity} & \textbf{Value} & \textbf{Relative Error} & \textbf{CPU time}\\
\midrule
Exact Value & 2.218878969089873 & --- & ---\\
GBFA Method Approximation    & 2.218878969089874 & $2.001412497195449 \times 10^{-16}$ & $0.003125$ s\\
($\lambda = \lambda_q = 0.5, n = 3, n_q = 4)$ &&&\\
MATLAB's \texttt{integral} Function Approximation & 2.218878973119478 & $1.816054080462689 \times 10^{-9}$ & $0.015625$ s\\
$(\text{RelTol} = \text{AbsTol} = 10^{-15})$ &&&\\
MATHEMATICA \texttt{NIntegrate} Function Approximation & 2.2188789399346960 & $1.3139598073477670 \times 10^{-8}$ & $0.0098979$ s\\
$(\text{PrecisionGoal} = \text{AccuracyGoal} = 15)$ &&&\\
\bottomrule
\label{tab:Tab1new}
\end{tabular}
}
\end{table}

\textbf{Example 4.} Consider the $\alpha$th-order RLFI of the function \( f(t) = \sin(1-t) \) over the interval \( \bm{\Omega}_1 \). The exact RLFI on \( \bm{\Omega}_t \) is given by
\[
\RLI{t}{\alpha}{f} = \frac{t^{\alpha}}{\Gamma(\alpha+1)} \sum_{k=0}^\infty \frac{(-1)^k}{\Gamma(2k + 2)} \, _2F_1\left( -2k - 1, 1, \alpha+1, t \right);
\]
cf. \cite{nowak2025neural}. In the latter work, a shallow neural network with $50$ hidden neurons was used to solve this problem. Trained on $1000$ random points $x_i \in [0, 1]$ with targets $\sin(1 - x_i)$, the network yielded a Euclidean error norm of $9.89 \times 10^{-3}$ for $\alpha = 0.2$, as reported in \cite{nowak2025neural}. The GBFA method, with $n = n_q = 16, \lambda = 1$, and $\lambda_q = 0.5$, achieved a much smaller error norm of about $7.63 \times 10^{-16}$ when evaluated over $1000$ equidistant points in $[0, 1]$.

\section{Conclusion and Future Works}
\label{sec:Conc}
This study presents the GBFA method, a powerful approach for approximating the left RLFI. By using precomputable FSGIMs, the method achieves super-exponential convergence for smooth functions, delivering near machine-precision accuracy with minimal computational effort. The tunable Gegenbauer parameters $\lambda$ and $\lambda_q$ enable tailored optimization, with sensitivity analysis showing that $\lambda_q < \lambda$ accelerates convergence for $n_q \sim n$. The strategic parameter selection guideline outlined in Section \ref{subsec:PGFPS1} improves the GBFA method's versatility, ensuring high accuracy for a wide range of problem types. Rigorous error bounds confirm rapid error decay, ensuring robust performance across diverse problems. Numerical results highlight the GBFA method’s superiority, surpassing MATLAB’s \texttt{integral}, MATHEMATICA’s \texttt{NIntegrate}, and prior techniques like trapezoidal \cite{dimitrov2021approximations}, spline-based methods \cite{ciesielski2024numerical}, and the neural network method \cite{nowak2025neural} by up to several orders of magnitude in accuracy, while maintaining lower computational costs. The FSGIM's invariance for fixed points improves efficiency, making the method ideal for repeated fractional integrations with varying $\alpha \in (0,1)$. These qualities establish the GBFA method as a versatile tool for fractional calculus, with significant potential to advance modeling of complex systems exhibiting memory and non-local behavior, and to inspire further innovations in computational mathematics. 

The error bounds derived in this study are asymptotic in nature, meaning they are most accurate for large values of $n$ and $n_q$. While these bounds provide valuable insights into the theoretical convergence behavior of the method, they may not be tight or predictive for small-to-moderate values of $n$ and $n_q$. For practitioners, this implies that while the asymptotic bounds guarantee eventual convergence, careful numerical experimentation is recommended to determine the optimal values of $n$, $n_q$, $\lambda$, and $\lambda_q$ for a given problem. The numerical experiments presented in Section \ref{sec:FNS} demonstrate that the method performs exceptionally well even for small and moderate values of $n$ and $n_q$, but the asymptotic bounds should be interpreted as theoretical guides rather than precise predictors for all cases. To address this limitation, future work could focus on deriving non-asymptotic error bounds or developing heuristic strategies for parameter selection in practical scenarios. Additionally, adaptive algorithms could be explored to dynamically adjust $n$ and $n_q$ based on local error estimates, further improving the method's robustness for real-world applications. While the current work focuses on $0 < \alpha < 1$, the extension of the GBFA method to $\alpha > 1$ is the subject of ongoing research by the authors and will be addressed in future publications. Exploring further adaptations to operators such as the Right RLFI presents exciting avenues for future work.

The super-exponential convergence of spectral and pseudospectral methods, including those using Gegenbauer polynomials and Gauss nodes, typically degrades to algebraic convergence when applied to functions with limited regularity. This degradation is a well-established phenomenon in approximation theory, with the rate of algebraic convergence directly related to the function's degree of smoothness. Specifically, for a function possessing $k$ continuous derivatives, theoretical analysis predicts a convergence rate of approximately $O(N^{-k})$, where $N$ represents the number of collocation points \cite{Gottlieb1995a}. To improve the applicability of the GBFA method for non-smooth functions, several methodological extensions can be incorporated: (i) Modal filtering techniques offer one promising approach, effectively dampening spurious high-frequency oscillations without significantly compromising overall accuracy. This process involves applying appropriate filter functions to the spectral coefficients, thereby regularizing the approximation while preserving accuracy in smooth solution regions. When implemented within the GBFA framework, filtering the coefficients in the SG polynomial expansion can potentially recover higher convergence rates for functions with isolated non-smoothness. The tunable parameters inherent in SG polynomials may provide valuable flexibility for optimizing filter characteristics to address specific types of non-smoothness. (ii) Adaptive interpolation strategies represent another valuable extension, employing local refinement near singularities or implementing moving-node approaches to more accurately capture localized features. By strategically concentrating computational resources in regions of limited regularity, these methods maintain high accuracy while effectively accommodating non-smooth behavior. Within the GBFA method, this approach could be realized through non-uniform SG collocation points with increased density near known singularities. (iii) Domain decomposition techniques offer particularly powerful capabilities by partitioning the computational domain into subdomains with potentially different resolutions or spectral parameters. This approach accommodates irregularities while preserving the advantages of the GBFA within each smooth subregion. Domain decomposition proves especially effective for problems featuring isolated singularities or discontinuities, allowing the GBFA method to maintain super-exponential convergence in smooth subdomains while appropriately addressing non-smooth regions through specialized treatment. (iv) For fractional problems involving weakly singular or non-smooth solutions with potentially unbounded derivatives, graded meshes provide an effective solution. These non-uniform discretizations concentrate points near singularities according to carefully chosen distributions, often recovering optimal convergence rates despite the presence of singularities. The inherent flexibility of SG parameters makes the GBFA method particularly amenable to such adaptations. (v)  Hybrid spectral-finite element approaches represent yet another viable pathway, combining the high accuracy of spectral methods in smooth regions with the flexibility of finite element methods near singularities. Such hybrid frameworks effectively balance accuracy and robustness for problems with limited regularity. The GBFA method could be integrated into these frameworks by utilizing its super-exponential convergence where appropriate while delegating singularity treatment to more specialized techniques. These theoretical considerations and methodological extensions can  significantly expand the potential applicability of the GBFA method to non-smooth functions commonly encountered in fractional calculus applications. While the current implementation focuses on smooth functions to demonstrate the method's super-exponential convergence properties, the framework possesses sufficient flexibility to accommodate various extensions for handling non-smooth behavior. Future research may investigate these techniques to extend the GBFA method's applicability to a wider range of practical problems in fractional calculus.

\vspace{6pt} 





\section*{Author Contributions}
The author confirms sole responsibility for the following: study conception and design, data collection, analysis and interpretation of results, and manuscript preparation.

\section*{Funding}
The author received no financial support for the research, authorship, and/or publication of this article.

\section*{Data Availability}
The author declares that the data supporting the findings of this study are available within the article.

\section*{Conflicts of Interest}
The author declares there is no conflict of interests.

\appendix
\section[\appendixname~\thesection]{List of Acronyms}\label{sec:LVA1}

\begin{table}[H]
    \centering
    \caption{List of Acronyms}
\resizebox{0.5\textwidth}{!}{%
    \begin{tabular}{|c|c|}
        \hline
        \textbf{Acronym} & \textbf{Meaning} \\
        \hline
        FSGIM & Fractional-order shifted Gegenbauer integration matrix \\
        GBFA & Gegenbauer-based fractional approximation\\
		RLFI & Riemann–Liouville fractional integral\\
        SGIRV & Shifted Gegenbauer integration row vector \\
        SG & Shifted Gegenbauer \\
        SGG & Shifted Gegenbauer-Gauss \\
        \hline
    \end{tabular}
    }
    \label{tab:acronyms}
\end{table}

\section{MATHEMATICAl Theorems and Proofs}
\label{sec:App1}
\begin{Theorem}\label{thm:dfjs1}
The function $T_1(\lambda) = (1 + 2\lambda)^{-1/2 - 2\lambda}$ is strictly concave over the interval $-0.5 < \lambda < 0$, and attains its maximum value at 
\[\lambda^* = \frac{-e + e^{\ProductLog(e/2)}}{2e} \approx -0.1351,\]
rounded to 4 significant digits.
\end{Theorem}
\begin{proof}
Notice that $\ln T_1(\lambda) = \left(-1/2 - 2\lambda\right)\ln(1 + 2\lambda)$. Thus,
\begin{align*}
\frac{d}{d\lambda}\ln T_1(\lambda) &= -2\ln(1 + 2\lambda) - \frac{1 + 4\lambda}{1 + 2\lambda},\\
\frac{d^2}{d\lambda^2}\ln T_1(\lambda) &= \frac{-6 - 8\lambda}{(1 + 2\lambda)^2}.
\end{align*}
For $\lambda \in \MBRmhzer, (1 + 2\lambda)^2 > 0$, and $-6 - 8\lambda$ is linear and decreasing in $\lambda$. Thus, $\frac{d^2}{d\lambda^2}\ln T_1(\lambda) < 0$ for all $\lambda \in (-0.5, 0)$. Since the logarithmic second derivative is negative, $\ln T_1(\lambda)$ is strictly concave on $(-0.5, 0)$. Since the logarithm is a strictly increasing function and $\ln T_1(\lambda)$ is strictly concave, then $T_1(\lambda)$ itself is also strictly concave on the interval $(-0.5, 0)$. This proves the first part of the theorem. Now, to prove that $T_1$ attains its maximum at $\lambda^*$, we set the logarithmic derivative equal to zero:
\[ -2\ln(1 + 2\lambda) - \frac{1 + 4\lambda}{1 + 2\lambda} = 0. \]
Let $x = 1 + 2\lambda$:
\[ -2\ln x - \frac{2x - 1}{x} = 0 \implies 2x\ln x + 2x - 1 = 0. \]
Let $x = e^u$: $2 u e^u + 2 e^u - 1 = 0 \implies e^{u+1}(u+1) = \tfrac{e}{2}$. The solution of this the transcendental equation is $u = \ProductLog(\tfrac{e}{2}) - 1$. Thus, $x = e^{\ProductLog(e/2)}/e = 1 + 2\lambda$, from which the maximum point $\lambda = \lambda^*$ is determined. This completes the proof of the theorem.
\end{proof}

Figure \ref{fig:Fig4} shows a sketch of $T_1$ on the interval $\MBRmhzer$.

\begin{figure}[ht]
\centering
\includegraphics[scale=0.75]{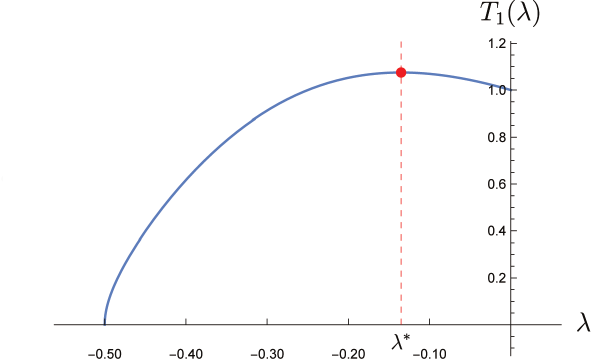}
\caption{Behavior of $T_1$ for $\lambda \in \MBRmhzer$. The red dashed line indicates $\lambda = \lambda^* \approx -0.1351$.}
\label{fig:Fig4}
\end{figure}

The following lemma establishes the asymptotic equivalence of $(n_q+1)! \hK_{n_q+1}^{\lambda_q}\,\forallL n_q$.
\begin{Lemma}\label{lem:LaElahaElaA1}
Let $\lambda_q > -\frac{1}{2}$ and consider the leading coefficient of the $(n_q+1)$st-degree, $\lambda_q$-indexed SG polynomial defined by
\[
\hK_{n_q}^{\lambda_q} = \frac{2^{n_q - 1} \Gamma(2\lambda_q + 1) \Gamma(n_q + \lambda_q)}{\Gamma(\lambda_q + 1) \Gamma(n_q + 2\lambda_q)};
\]
cf. \cite[see p.g. 103]{elgindy2018high}. Then,
\[
(n_q + 1)! \, \hK_{n_q + 1}^{\lambda_q} \sim \Theta_{\lambda_q} n_q^{3/2 - \lambda_q} \left( \frac{2n_q}{e} \right)^{n_q}\quad \forallL n_q,
\]
where 
\begin{equation}\label{eq:qwqwq2}
\Theta_{\lambda_q} = \frac{\sqrt{2\pi} \, \Gamma(2\lambda_q + 1)}{\Gamma(\lambda_q + 1)}.
\end{equation}
\end{Lemma}
\begin{proof}
We begin by expressing $\hK_{n_q + 1}^{\lambda_q}$ explicitly:
\[
\hK_{n_q + 1}^{\lambda_q} = \frac{2^{n_q} \Gamma(2\lambda_q + 1) \Gamma(n_q + \lambda_q + 1)}{\Gamma(\lambda_q + 1) \Gamma(n_q + 2\lambda_q + 1)}.
\]
Since
\begin{equation}
(n+1)! \hK_{n_q+1}^{\lambda_q} = O\left(n_q^{3/2-\lambda_q} \left(\frac{2 n_q}{e}\right)^{n_q}\right),\quad \text{as }n_q \to \infty, 
\end{equation}
$\foralla \lambda_q > -1/2$ by \cite[Lemma 2.2]{elgindy2013solving} and \cite[Lemma 4.2]{Elgindy20161}. Then,
\begin{equation}\label{eq:alma3arej1}
\frac{(n_q + 1)! \, \hK_{n_q + 1}^{\lambda_q}}{n_q^{3/2 - \lambda_q} \left( \frac{2n_q}{e} \right)^{n_q}} = \frac{(n_q + 1)! e^{n_q}\, \Gamma(2\lambda_q + 1) \Gamma(n_q + \lambda_q + 1)}{\Gamma(\lambda_q + 1) \Gamma(n_q + 2\lambda_q + 1)\,n_q^{n_q - \lambda_q + 3/2}}.
\end{equation}
Stirling's approximations to the factorial and gamma functions $\,\forallL n_q$ give
\begin{align*}
(n_q + 1)! &\approx \sqrt{2\pi}\,n_q^{n_q + 3/2} e^{-n_q}, \\
\Gamma(n_q + \lambda_q + 1) &\approx \sqrt{2\pi}\,n_q^{n_q + \lambda_q + 1/2} e^{-n_q}, \\
\Gamma(n_q + 2\lambda_q + 1) &\approx \sqrt{2\pi}\,n_q^{n_q + 2\lambda_q + 1/2} e^{-n_q}.
\end{align*}
Therefore,
\begin{equation}
\frac{(n_q + 1)! \Gamma(n_q + \lambda_q + 1)}{\Gamma(n_q + 2\lambda_q + 1)} = \sqrt{2\pi} n_q^{n_q - \lambda_q + 3/2} e^{-n_q}.
\end{equation}
Multiplying by the remaining terms:
\[
\frac{\Gamma(2\lambda_q + 1) e^{n_q}}{\Gamma(\lambda_q + 1) n_q^{n_q - \lambda_q + 3/2}} \cdot \sqrt{2\pi} n_q^{n_q - \lambda_q + 3/2} e^{-n_q} = \Theta_{\lambda_q}.
\]
This shows that 
\[
\frac{(n+1)! \hK_{n_q+1}^{\lambda_q}}{n_q^{3/2-\lambda_q} \left(\frac{2 n_q}{e}\right)^{n_q}} \sim \Theta_{\lambda_q}\quad \forallL n_q,
\]
which completes the proof.
\end{proof}

The next lemma is needed for the proof of the following lemma.

\begin{Lemma}[Falling Factorial Approximation]\label{lem:nmnm1}
Let $n = m + k: k = o(m)\,\forallL m \in \MBZP$. Then the falling factorial satisfies:
\[
(n)_m = m^m \left(\frac{m}{k}\right)^{k+1/2} e^{k-m-\tfrac{1}{4k}-\tfrac{k^2}{2m}} \exp\left(O\left(\max\left(\tfrac{k^3}{m^2}, \tfrac{1}{k}\right)\right)\right).
\]
\end{Lemma}
\begin{proof}
Express the falling factorial as
\[
(n)_m = \prod_{i=0}^{m-1} (m + k - i) = \prod_{j=1}^m (k + j).
\]
Take logarithms to convert the product into a sum:
\[
\ln (n)_m = \sum_{j=1}^m \ln(k + j).
\]
Since $m$ is large, we can approximate the sum by using the midpoint rule of integrals:
\begin{gather}
\sum_{j=1}^m \ln(k + j) \approx \C{I}_{1/2,m+1/2}^{(x)} {\ln(k + x)}\\
= (k+m+\tfrac{1}{2})\ln(k+m+\tfrac{1}{2}) - (k+\tfrac{1}{2})\ln(k+\tfrac{1}{2}) - m.
\end{gather}
For $k = o(m)$ and large $m$, Taylor series expansions produce:
\begin{align*}
&(k+m+\tfrac{1}{2}) \ln(k+m+\tfrac{1}{2}) = (k+m+\tfrac{1}{2}) \ln\left(m \left(1+\frac{k+\tfrac{1}{2}}{m}\right)\right)\\
&\approx m \ln m + m \left[\frac{k+\tfrac{1}{2}}{m} - \frac{1}{2}\left(\frac{k+\tfrac{1}{2}}{m}\right)^2 + O\left(\left(\frac{k+\tfrac{1}{2}}{m}\right)^3\right)\right]\\
&= m\ln m + k + \tfrac{1}{2} - \tfrac{k^2}{2m} + O\left(\tfrac{k^3}{m^2}\right) + O\left(\tfrac{1}{m}\right),
\end{align*}
and
\begin{align*}
&(k + {\textstyle{1 \over 2}})\ln (k + {\textstyle{1 \over 2}}) = (k + {\textstyle{1 \over 2}})\left[ {\ln k + \ln (1 + {\textstyle{1 \over {2k}}})} \right]\\
&= \tfrac{1}{2} + \tfrac{1}{{4k}} + (k+\tfrac{1}{2})\ln k + O\left( {\frac{1}{k}} \right).
\end{align*}
Substituting these approximations back yields:
\begin{gather}
\ln (n)_m \approx m\ln m - k\ln k + k - \tfrac{1}{2}\ln k - \tfrac{1}{4k} - m - \tfrac{k^2}{2m} + O\left(\tfrac{k^3}{m^2}\right) + O\left(\tfrac{1}{k}\right).
\end{gather}
Exponentiating gives:
\begin{gather}
(n)_m \approx m^m k^{-k} e^{k-m} k^{-1/2} e^{-1/(4k)} \exp\left(-\tfrac{k^2}{2m} + O\left(\tfrac{k^3}{m^2}\right) + O\left(\tfrac{1}{k}\right)\right) \\
= m^{m-k} \left(\tfrac{m}{k}\right)^k e^{k-m} \sqrt{\frac{m}{k}} \frac{1}{\sqrt{m}} e^{-1/(4k)} \exp\left(-\tfrac{k^2}{2m} + O\left(\tfrac{k^3}{m^2}\right) + O\left(\tfrac{1}{k}\right)\right)\\
= m^m \left(\frac{m}{k}\right)^{k+1/2} e^{k-m-\tfrac{1}{4k}-\tfrac{k^2}{2m}} \exp\left(O\left(\max\left(\tfrac{k^3}{m^2}, \tfrac{1}{k}\right)\right)\right).
\end{gather}
\end{proof}

The next lemma gives the asymptotic upper bound on the growth rate of the parameteric scalar $\hchi_{n,m}^{\lambda}$.

\begin{Lemma}\label{lem:2}
Let $\lambda > -1/2$ and $m \in \MBZP: m \gg \lambda$. Then
\begin{align}
\hchi_{n,m}^{\lambda} = \left\{ \begin{array}{l}
O\left(\left(\frac{4}{e}\right)^m\,{m^{m + \frac{1}{2} - \lambda }}\right)\quad \text{if }n \sim m,\\
O\left(n^{2m}\right)\quad \forallL n \gg m.
\end{array} \right.
\end{align}
\end{Lemma}
\begin{proof}
Let $n \sim m$. Then $n = m + k: k = o(m)$, and we can write
\[
(n)_m \approx m^{m+1/2} e^{- m},
\]
by Lemma \ref{lem:nmnm1}. Substituting this approximation and the sharp inequalities of the Gamma function \cite[Ineq. (96)]{elgindy2018optimal} into Eq. \eqref{eq:hhkk1} give:
\begin{gather}
\chi(n, m, \lambda) \simlt \hat{\ell}_{\lambda}\,m^{m+1/2} e^{- m} \cdot {m^{ - m - \lambda }}{n^{ - n - 2\lambda  + \frac{1}{2}}}{\left( {m + n} \right)^{m + n + 2\lambda  - \frac{1}{2}}}\notag\\
 \sim {\ell _\lambda }\,\left(\frac{4}{e}\right)^m\,{m^{m + \frac{1}{2} - \lambda }},
\end{gather}
where $\ell_{\lambda} = \tfrac{4^{\lambda}}{\sqrt{2}} \hat{\ell}_{\lambda}$, and $\hat{\ell}_{\lambda} \in \MBRP$ is a $\lambda$-dependent constant. Now, consider the case when $n \gg m$. Applying Stirling's approximation of factorials and the same sharp inequalities of the Gamma function on the parameteric scalar $\hchi_{n,m}^{\lambda}$ give:
\begin{gather}
\hchi_{n,m}^{\lambda} \simlt h_m^{\lambda} {n^{\frac{1}{2} + n}}{\left( {n - m} \right)^{\frac{1}{2}\left( { - 1 + 3m - 3n} \right)}}{\left( {n + 2\lambda - 1 } \right)^{\frac{1}{2} - n - 2\lambda }} \times\\
{\left( {n + m + 2\lambda  - 1} \right)^{n+m + 2\lambda - \frac{1}{2}}}\sinh^{\frac{{m - n}}{2}}{\left( {\frac{1}{{n - m}}} \right)} \times\\
{\left[ {\left( {n + 2\lambda  - 1 } \right)\sinh \left( {\frac{1}{{n + 2\lambda - 1  }}} \right)} \right]^{\frac{1}{2}\left( {1 - n - 2\lambda } \right)}} \times\\
{\left[ {\left( {n + m + 2\lambda  - 1} \right)\sinh \left( {\frac{1}{{n + m + 2\lambda - 1 }}} \right)} \right]^{\frac{1}{2}\left( {n+m - 1} \right) + \lambda }}\\
\sim \hbar_m^{\lambda} n^{\frac{1}{2} (5 m - n)} \sinh^{\frac{{m - n}}{2}}{\left( {\frac{1}{{n - m}}} \right)},\label{eq:bbjsf1}
\end{gather}
where $\hbar_m^{\lambda} = e^{5 m/2} h_m^{\lambda}$, and $h_m^{\lambda}$ is a constant dependent on the parameters $m$ and $\lambda$. Since $1/(n-m)$ is small $\forallL n$, we have
\begin{equation}\label{eq:pop1}
\sinh^{\frac{{m - n}}{2}}{\left( {\frac{1}{{n - m}}} \right)} \approx \left( {\frac{1}{{n - m}}} \right)^{\frac{{m - n}}{2}} \sim n^{\frac{n-m}{2}}\quad \forall n \gg m.
\end{equation}
Substituting \eqref{eq:pop1} into \eqref{eq:bbjsf1} yields
\[\hchi_{n,m}^{\lambda} \simlt \hbar_m^{\lambda} {n^{2m}},\]
from which the proof is accomplished.
\end{proof}

\bibliographystyle{model1-num-names}
\bibliography{Bib}
\end{document}